\documentclass[12pt, twoside, fleqn, a4paper]{article}

\usepackage{latexsym}  
\usepackage{amscd}    
\usepackage{theorem}  
\usepackage{pifont}  
\usepackage{mathbbol} 
\usepackage{amsfonts}  
\usepackage{xspace}  
\usepackage{amssymb} 
\usepackage{fancyhdr} 
\usepackage{mathrsfs} 
\usepackage{amsmath} 
\usepackage{graphics} 
\usepackage{graphicx} 
\usepackage{euscript}
\usepackage{psfrag}
\usepackage{empheq} 
\usepackage{mathabx} 
\usepackage[parfill]{parskip}     
\usepackage[colorlinks=true, allcolors=blue]{hyperref}
\usepackage{bigints}

\title{On induced systems and ergodicity for holomorphic correspondences} 
\author{Sathi Trikkadeeri Mana\footnote{Indian Institute of Science Education and Research Thiruvananthapuram (IISER-TVM), Maruthamala P.O., Vithura, Kerala, India. PIN 695 551.\ \ ORCID : 0009-0000-1444-4604, email: \texttt{sathitm23@iisertvm.ac.in}}\ \ and\ \ Shrihari Sridharan\footnote{Indian Institute of Science Education and Research Thiruvananthapuram (IISER-TVM), Maruthamala P.O., Vithura, Kerala, India. PIN 695 551.\ \ ORCID : 0000-0003-2434-4767, email: \texttt{shrihari@iisertvm.ac.in} (Corresponding Author). \\ The second named author thanks the support provided by NBHM through a research grant \\ No. 02011/35/2025/NBHM(R.P)/R\&D II/9832}}

\DeclareFontFamily{OT1}{pzc}{}
\DeclareFontShape{OT1}{pzc}{m}{it}%
              {<-> s * [0.900] pzcmi7t}{}
\DeclareMathAlphabet{\mathpzc}{OT1}{pzc}%
                                 {m}{it}

{\theorembodyfont{\slshape} \newtheorem{theorem}{Theorem}[section]}
{\theorembodyfont{\slshape} \newtheorem{definition}[theorem]{Definition}}
{\theorembodyfont{\slshape} \newtheorem{lemma}[theorem]{Lemma}}
{\theorembodyfont{\slshape} \newtheorem{proposition}[theorem]{Proposition}}
{\theorembodyfont{\slshape} }
{\theorembodyfont{\slshape} \newtheorem{property}[theorem]{Property}}
{\theorembodyfont{\slshape} \newtheorem{corollary}[theorem]{Corollary}} 
{\theorembodyfont{\slshape} } 
{\theorembodyfont{\slshape} \newtheorem{remark}[theorem]{Remark}}
{\theorembodyfont{\slshape} \newtheorem{Example}[theorem]{Example}}

\numberwithin{equation}{section}

\newenvironment{proof}{\paragraph{Proof:}}{\hfill$\blacktriangle$}
\newenvironment{example}{\begin{Example}}{\hfill$\triangle$\end{Example}}

\begin{document}

\maketitle 

\begin{abstract} 
The Poincar\'{e} recurrence theorem for a holomorphic correspondence provides a natural framework for studying the first return time and a dynamical system thus induced, in this setting. After establishing a version of the Kac's theorem on the average of the first return times for a holomorphic correspondence, we introduce the concept of an induced correspondence, akin to the case of dynamics of maps and prove that it inherits the ergodicity of the holomorphic correspondence. Additionally, we characterise the ergodic measures with respect to the considered holomorphic correspondence, establish the density of forward orbits under mild conditions and prove a Rokhlin-Kakutani type lemma. 
\end{abstract}

\begin{tabular}{l l} 
{\bf Keywords} & Holomorphic correspondences\\ 
& Recurrence\\
& Kac's theorem\\
& Induced correspondences \\
& Ergodicity \\
\\
{\bf MSC Subject} & \\ 
{\bf Classifications} & 37F05, 37B20, 37A25 \\ 
\end{tabular} 
\bigskip 

\section{Introduction}
\label{sec:intro}
The study of recurrence in dynamical systems has been a central theme in the development of modern dynamics since the pioneering work of Henri Poincar\'{e}. The celebrated Poincar\'{e} recurrence theorem states that for certain self-maps of a probability space, the orbit of almost every point in a set of positive measure visits that set infinitely often \cite{poinc:1890}. This was formalised in 1890 during his research on the stability of three-body problem in celestial mechanics. Subsequently in 1947, Kac extended the framework of the Poincar\'{e} recurrence theorem by establishing a fundamental quantitative refinement. He proved that for a measure-preserving dynamical system, the expected return time to a measurable set is precisely equal to the reciprocal of the measure of that set. The theorem can be formally stated as follows.  

\begin{theorem} \cite{kac:1947} 
\label{thm:kacfort} 
Let $(X, \mathscr{B}, M, T)$ be an ergodic measure preserving system, meaning $X$ is a compact probability space, $\mathscr{B}$ a Borel $\sigma$-algebra of subsets of $X$, $M$ a probability measure supported on $X$ and $T$ a transformation defined on $X$ that preserves the measure $M$ and is ergodic with respect to the measure. Let $E \subseteq X$ be such that $M(E) > 0$. Define the first return time of a point $x_{0} \in E$ to the set $E$ as 
\begin{eqnarray}
\label{eqn:1}
n_{E} \left( x_{0} \right)\ \ =\ \ \inf \left\{ p\, \in\, \mathbb{Z}_{+}\; :\; T^{p} \left( x_{0} \right) \in E \right\}. 
\end{eqnarray}
Then, the average of the first return times satisfies 
\[ \int_{E} n_{E}\,\mathrm{d}M_{E}\ \ =\ \ \frac{1}{M (E)},\ \ \ \ \text{where}\ \ M_{E} ( E' )\ \ =\ \ \dfrac{M \left( E \cap E' \right)}{M (E)}\ \ \ \text{for any Borel subset}\ E' \subseteq X. \] 
\end{theorem} 

This naturally leads to defining the induced map $T_{E}(x) = T^{n_{E}(x)}(x)$, for $M$-almost every point $x \in E$. By tracking the system only at return times, this map compresses the dynamics to $E$, for $M$ - almost every point $x \in E$, yielding a full dynamical system supported on $E$. Building upon these, further research has focused on dynamical systems induced by first return times. 

The central object of our work is what is known as a holomorphic correspondence defined on $\widehat{\mathbb{C}}$, which exhibits rich analytical and dynamical properties. The study of iterated holomorphic correspondences originated with the works of Bullett and his collaborators, \cite{bul:1988, BH:2000, BF:2005}, in particular the research articles due to Bullett and Penrose \cite{BP:1994, BP:_1994, BP:2001}, where the authors developed a general theory of Julia sets, alongside `mating' correspondences for quadratic polynomials and the modular group. Thereupon, Dinh and Sibony in \cite{dinh:2005, ds:2005, ds:2006} introduced certain fundamental ergodic tools to the system, by establishing canonical invariant measures and entropy bounds for meromorphic correspondences. Further, Bullett, Lomonaco, Lyubich, Mukherjee, Mahan and Mazor in \cite{LMM:2024, MM:2025, LM:2025, BLLM:2026} developed a unifying framework for mating rational maps with Kleinian and Fuchsian groups via algebraic correspondences. Works of Bharali and Sridharan \cite{bs:2016, bs:2017} regarding the concepts of normality in the family of iterates of a holomorphic correspondence further developed the field. Subsequent works of Londhe concerning recurrence and ergodicity of a holomorphic correspondence, as in \cite{lon:2022, lon:2024} provide a conceptual motivation for our present study. In \cite{lon:2022}, the author establishes an analogue of the Poincar\'{e} recurrence theorem for the system under consideration, thereby providing a crucial fundamental result in this setting. This naturally motivates the present investigation into the return times and the associated induced systems. In \cite{lon:2024}, the author introduced a framework for ergodicity and proved an analogue of Birkhoff's ergodic theorem. Using these foundation results, we extend the analysis to examine certain other interesting aspects of ergodicity.

This manuscript is divided into two parts : The first part of this work deals with the return time function defined for a point, as a consequence of the analogue of Poincar\'{e} recurrence theorem for the system. We prove a version of Kac's theorem, under certain assumptions, and discuss some special cases and interesting examples. Motivated by these developments, we introduce and analyse a class of multivalued maps, namely the induced correspondences. We establish that the induced system inherits ergodicity from the underlying holomorphic correspondence, thereby providing a structural link between the dynamics of the original correspondence and the induced system. The second part of the study concerns the ergodicity of holomorphic correspondences, as defined by Londhe. We give some interesting equivalences and density arguments related to the same. 

This article is organised as follows. In Section \ref{sec:prelims} and Section \ref{sec:orbits}, we develop the machinery required for our study. Beginning with the definition of a holomorphic correspondence, we formulate the associated set-valued dynamical system, introduce the notion of invariant measures and construct the corresponding orbit space. Section \ref{sec:recc} provides a concise overview of the major ideas and results from \cite{lon:2022,lon:2024}, which will be instrumental in our subsequent analysis. The main results of this work are listed in Section \ref{sec:results}. In Section \ref{sec:kac_proof}, we establish a version of Kac's theorem for the concerned system, complemented by several special cases and illustrative examples. Section \ref{sec:induced} is devoted to the introduction of an induced system, based on the first return times of points, one of the central objects of investigation in this work. We examine certain structural and dynamical properties of the induced system and prove that it inherits ergodicity from the underlying holomorphic correspondence (Theorem \ref{thm:erg}). In Section \ref{sec:erg}, we prove the remainder of the main results of this paper that examine certain ergodic aspects of a holomorphic correspondence. We begin by proving a result that characterises ergodicity in Theorem \ref{thm:equierg}, followed by Theorem \ref{thm:dense} which establishes the density of forward orbits under mild assumptions. We conclude with the proof of Theorem \ref{thm:rkforc}, an analogue of the Rokhlin-Kakutani lemma for the system.

\section{Preliminaries}
\label{sec:prelims}

This section is intended to familiarise us with the foundational concepts in the dynamics of holomorphic correspondences. We begin this section with the definition of a holomorphic correspondence, a fundamental object that we shall be interested in, throughout this paper. 

\begin{definition} 
\label{defn:holcorresp} 
A \emph{holomorphic correspondence} $\Gamma \subseteq \widehat{\mathbb{C}} \times \widehat{\mathbb{C}}$ is given by a formal linear combination of the form 
\begin{equation} 
\label{eqn:holo} 
\Gamma\ \ =\ \ \sum_{\alpha\, =\, 1}^{N} m_{\alpha} \Gamma_{\alpha}, 
\end{equation} 
where $m_{\alpha}$’s are positive integers that count the respective multiplicities of the distinct irreducible complex-analytic subvarieties $\Gamma_{\alpha}$'s of dimension $1$ that satisfy the following conditions: For $i = 1, 2$, suppose $\pi_{i}$ denotes the projection onto the $i$-th coordinate, then 
\begin{enumerate} 
\item $\pi_{i} \vert_{\Gamma_{\alpha}}$ is surjective for each $i = 1, 2$ and for every $1 \le \alpha \le N$ and 
\item the sets $\pi_{1} \left( \pi_{2}^{-1} \left( \left\{ z_{0} \right\} \right)\cap \Gamma_{\alpha} \right)$ and $\pi_{2} \left( \pi_{1}^{-1} \left( \left\{ z_{0} \right\} \right)\cap \Gamma_{\alpha} \right)$ are of finite cardinality for any generic point $z_{0} \in \widehat{\mathbb{C}}$ and for every $1 \le \alpha \le N$. 
\end{enumerate} 
\end{definition}

Throughout this paper, there are instances where we use the term ``set" to represent the collection of all images and pre-images of points in $\widehat{\mathbb{C}}$, counted with multiplicities, under $\Gamma$, as we have done in the second condition of Definition \ref{defn:holcorresp}. We denote the set-valued map stemming from the holomorphic correspondence $\Gamma$ by $\digamma$ and call the set $\pi_{2} \left( \pi_{1}^{-1} \left( \left\{ z_{0} \right\} \right)\cap \Gamma \right)$ as the set of images of the point $z_{0}$ under $\Gamma$, denoted by $\digamma (z_{0})$. For a generic point $z_{0}$, the cardinality of the set $\digamma (z_{0})$ is called the \emph{forward degree} of $\Gamma$ and is denoted by $d_{{\rm fwd}}$.

Suppose $\Gamma^{\dagger}$ denotes the adjoint of the holomorphic correspondence $\Gamma$, written in Equation \eqref{eqn:holo} given by 
\begin{equation} 
\label{eqn:adjholocorr} 
\Gamma^{\dagger} \ \ =\ \ \sum_{\alpha\, =\, 1}^{N} m_{\alpha} \Gamma_{\alpha}^{\dagger},\ \ \text{where}\ \ (w, z) \in \Gamma_{\alpha}^{\dagger}\ \text{whenever}\ (z, w) \in \Gamma_{\alpha}.
\end{equation} 
Then, it is clear from Definition \ref{defn:holcorresp} and Equation \eqref{eqn:adjholocorr} that $\Gamma^{\dagger}$ is a holomorphic correspondence, in its own right. Suppose $\digamma^{\dagger}$ denotes the set-valued map stemming from the holomorphic correspondence $\Gamma^{\dagger} $. Then $\digamma^{\dagger} (z_{0}) = \pi_{1} \left( \pi_{2}^{-1} \left( \left\{ z_{0} \right\} \right)\cap \Gamma_{\alpha} \right)$ represents the set of pre-images of $z_{0}$ under $\Gamma$. The cardinality of the set $\digamma^{\dagger} (z_{0})$, for a generic point $z_0$, is called the \emph{topological degree} of the holomorphic correspondence $\Gamma$ and is denoted by $d_{{\rm top}}$.

The set of images and pre-images of a set $E \subseteq \widehat{\mathbb{C}}$ with respect to a holomorphic correspondence $\Gamma$ is given by $\displaystyle{\digamma (E) = \bigcup_{z_{0}\, \in\, E}\digamma(z_{0})}$ and $\displaystyle{\digamma^{\dagger} (E)=\bigcup_{z_{0}\, \in\, E}\digamma^{\dagger}(z_{0})}$, respectively. Note that the sets $\digamma (E)$ and $\digamma^{\dagger} (E)$ are Borel, whenever $E \subseteq \widehat{\mathbb{C}}$ is a Borel set, (see \cite{lon:2024}). The set underlying the representation of $\Gamma$, called \emph{the support of $\Gamma$}, is given by $\displaystyle{\big| \Gamma \big| = \bigcup_{\alpha\, =\, 1}^{N} \Gamma_{\alpha}}$. The ability to compose two correspondences introduces a perspective of dynamics to the study of correspondences. The following definition illustrates a way to compose any two holomorphic correspondences $\Gamma^{1}$ and $\Gamma^{2}$ on $\widehat{\mathbb{C}}$, see \cite{bs:2016, bs:2017}. 

\begin{definition} 
Suppose $\Gamma^{1} = \sum_{1\, \leq\, \alpha\, \leq\, N_{1}}' \Gamma_{\alpha}^{1}$ and $ \Gamma^{2} = \sum_{1\, \leq\, \beta\, \leq\, N_{2}}' \Gamma_{\beta}^{2}$ are two holomorphic correspondences on $\widehat{\mathbb{C}}$, where the primed sums indicate the repetition of the varieties according to its multiplicity. Then, $\big| \Gamma^{2} \circ \Gamma^{1} \big|$ is merely the set obtained by the classical composition of $\left|\Gamma^{2}\right|$ with $\left|\Gamma^{1}\right|$ as relations. If $Y_{s, \alpha \beta},\ s = 1, \cdots, N(\alpha, \beta)$ are the distinct irreducible components of $\left|\Gamma_{\beta}^{2}\right| \circ \left|\Gamma_{\alpha}^{1}\right|$, then let $\eta_{s, \alpha \beta} = \# \left\{ y \in \widehat{\mathbb{C}} : \text{for any} \, \, (x, z) \in Y_{s, \alpha\beta}, \text{we have} \, \, (x, y) \in \Gamma_{\alpha}^{1}, (y, z) \in \Gamma_{\beta}^{2} \right\}$.  
This entails the following definition for the  composition of the two holomorphic correspondences: 
\[ \Gamma^{2} \circ \Gamma^{1}\ \ =\ \ \sum_{1\, \leq\, \beta \, \leq\, N_{2}}\ \sum_{1\, \leq\, \alpha \, \leq\, N_{1}}\ \sum_{1\, \leq\, s\, \leq\, N(\alpha, \beta)}\ \eta_{s, \alpha \beta} Y_{s, \alpha \beta}.\] 
\end{definition}

Suppose we consider $\Gamma$ to be a holomorphic correspondence on $\widehat{\mathbb{C}}$, then the above rule for composition of correspondences yields a natural definition for iteration of $\Gamma$ with itself. For $n \in \mathbb{Z}_{+}$, we denote by $\Gamma^{\circ n}$, the $n$-fold composition of $\Gamma$ with itself, {\it i.e.}, $\Gamma^{\circ n} = \Gamma \circ \Gamma \circ \cdots \circ \Gamma$, which is also a holomorphic correspondence. We denote by $\digamma^{n}$, the set-valued map stemming from the holomorphic correspondence $\Gamma^{\circ n}$. Throughout the paper, we shall use the notations $\Gamma$ and $\digamma$ alternatively, that one can learn from the context whether we refer to the correspondence or the set-valued map stemming from the correspondence. 

Let $\mathscr{M} \left( \widehat{\mathbb{C}} \right)$ denote the set of all Borel probability measures supported on $\widehat{\mathbb{C}}$. Then, for every measure $\mu \in \mathscr{M} \left( \widehat{\mathbb{C}} \right)$, the holomorphic correspondence $\digamma$ induces a pullback measure denoted by $\digamma^{*} \mu \in \mathscr{M} \left( \widehat{\mathbb{C}} \right)$ given by $\displaystyle{\int_{\widehat{\mathbb{C}}} f\, \mathrm{d} \left(\digamma^{*} \mu \right) =\int_{\widehat{\mathbb{C}}} \sum_{w\, \in\, \digamma^{\dagger} (z)} f(w)\, \mathrm{d}\mu\  \text{for any} \, f \in \mathcal{C} \left( \widehat{\mathbb{C}}, \mathbb{R} \right)}$. 

\begin{definition} 
A measure $\mu \in \mathscr{M} \left( \widehat{\mathbb{C}} \right)$ is said to be \emph{$\digamma^{*}$-invariant} if $\mu (E) = \dfrac{1}{d_{{\rm top}}} \left( \digamma^{*} \mu \right) (E)$ for every Borel set $E \subseteq \widehat{\mathbb{C}}$, or equivalently 
\[ \int_{\widehat{\mathbb{C}}} f\, \mathrm{d} \mu\ \ =\ \ \frac{1}{d_{{\rm top}}} \int_{\widehat{\mathbb{C}}} f\, \mathrm{d} \left(\digamma^{*} \mu \right)\ \ \ \ \ \text{for any}\ f \in \mathcal{C} \left( \widehat{\mathbb{C}}, \mathbb{R} \right). \] 
\end{definition} 

The density of $\mathcal{C} \left( \widehat{\mathbb{C}}, \mathbb{R} \right)$ in $\mathscr{L}^{p}(\widehat{\mathbb{C}}, \mu)$ for $1 \le p < \infty$ implies that the above equation remains valid for any $f \in \mathscr{L}^{p}(\widehat{\mathbb{C}}, \mu)$. We now provide an example of a $\digamma^{*}$-invariant measure in $\mathscr{M} \left( \widehat{\mathbb{C}} \right)$, by considering an amalgamation of results from \cite{ds:2006, bs:2016}. 

\begin{example}\cite{ds:2006, bs:2016} 
\label{ex:invmeas} 
Let $\digamma$ be a holomorphic correspondence defined on $\widehat{\mathbb{C}}$. 
\begin{enumerate} 
\item Let $d_{{\rm top}} > d_{{\rm fwd}}$. Then, there exists an exceptional set $\mathcal{E}$ such that the sequence of measures $\displaystyle{\left\{ \frac{1}{\left( d_{{\rm top}} \right)^{n}} \left( \digamma^{n} \right)^{*} \delta_{z_{0}} \right\}_{n\, \in\, \mathbb{Z}_{+}}}$ converges (in the weak*-topology) to some measure, independent of the chosen point $z_{0} \in \widehat{\mathbb{C}} \setminus \mathcal{E}$. 
\item Suppose $d_{{\rm top}} \le d_{{\rm fwd}}$. Assume that $\digamma$ has a strong repeller $\mathcal{R}$ that is disjoint from the set of critical values of $\digamma$. Then there exists an open set $U \left( \digamma, \mathcal{R} \right) \supset \mathcal{R}$ such that the sequence of measures $\displaystyle{\left\{ \frac{1}{d_{top}^{n}} \left( \digamma^{n} \right)^{*} \delta_{z_{0}} \right\}_{n\, \in\, \mathbb{Z}_{+}}}$ converges (in the weak*-topology) to some measure, independent of the chosen point $z_{0} \in U \left( \digamma, \mathcal{R} \right)$. 
\end{enumerate} 

Observe that the limiting measure in both the statements above quantifies the distribution of pre-images of the chosen point $z_{0}$. This measure is called the \emph{Dinh-Sibony measure}, denoted by $\omega_{{\rm DS}}$ and is an example of a $\digamma^{*}$-invariant measure in $\mathscr{M} \left( \widehat{\mathbb{C}} \right)$. Note that $\omega_{{\rm DS}}$ is analogous to the limiting measure due to Brolin for polynomial maps, as one may find in \cite{bro:1965} or the limiting measure due to Lyubich for a rational map restricted on its Julia set, as one may find in \cite{lju:1983}. Suppose the support of the measure $\omega_{{\rm DS}}$ is denoted by $\Omega_{{\rm DS}}$, then for any $n \in \mathbb{Z}_{+}$, we have $(\digamma^{n})^{\dagger} \left( \Omega_{{\rm DS}} \right) \subseteq \Omega_{{\rm DS}}$. Even though $\Omega_{{\rm DS}}$ is not forward invariant with respect to $\digamma$, we have $\digamma^{n} \left( \Omega_{{\rm DS}} \right) \cap \Omega_{{\rm DS}} \ne \emptyset$, for every $n \in \mathbb{Z}_{+}$ (see \cite{lon:2022}). 
\end{example} 

\section{The space of bi-infinitely long paths} 
\label{sec:orbits}

In this section, we present another perspective of looking at the dynamics of a holomorphic correspondence $\Gamma$ by looking at the paths of points, which apart from keeping track of the orbit of points in $\widehat{\mathbb{C}}$ also takes into account the varieties through which the orbit evolves in either direction. Towards that end, we define the collection of all orbits of the point $z_{0} \in \widehat{\mathbb{C}}$ with respect to $\Gamma$ as
\begin{eqnarray} 
\label{eqn:di_gamma}
\mathscr{O}^{\Gamma} \left( z_{0} \right) & = & \Bigg\{ \left( \cdots, z_{k_{2}}^{(-2)}, z_{k_{1}}^{(-1)}, \underline{z_{0}}, z_{j_{1}}^{(1)}, z_{j_{2}}^{(2)}, \cdots;\; \cdots, \beta_{2}, \beta_{1} \vert \alpha_{1}, \alpha_{2}, \cdots \right)\ :\ \left( z_{j_{i - 1}}^{(i - 1)}, z_{j_{i}}^{(i)} \right) \in \Gamma_{\alpha_{i}} ,\nonumber \\ 
& & \hspace{+1cm} \left( z_{k_{i}}^{(-i)}, z_{k_{i - 1}}^{(-(i - 1))}\right) \in \Gamma_{\beta_{i}}\ \text{where}\ z_{j_{0}}^{(0)}=z_{k_{0}}^{(0)} = z_{0}, \,  \alpha_{i}, \beta_{i} \in \left\{ 1, 2, \cdots, N \right\} \ \ \text{and} \nonumber \\ 
 & & \hspace{+1cm}  1 \le j_{i} \le m_{\alpha_{i}} \lambda_{\alpha_{i}} \left( z_{j_{i - 1}}^{(i - 1)} \right) , \ 1 \le k_{i} \le m_{\beta_{i}} \delta_{\beta_{i}} \left( z_{k_{i - 1}}^{(-(i - 1))} \right)\ \text{for}\ i \in \mathbb{Z}_{+} \Bigg\} 
\end{eqnarray} 
where $\lambda_{\alpha_{i}} \left( z^{(i - 1)}_{j_{i - 1}} \right)$ and $\delta_{\beta_{i}}\left( z^{(-(i - 1))}_{k_{i - 1}} \right)$ denotes the cardinality of the sets $\left\{z^{(i)}_{j_{i}}: \left( z^{(i - 1)}_{j_{i - 1}}, z^{(i)}_{j_{i}} \right) \in \Gamma_{\alpha_{i}} \right\}$ and  $\left\{z^{(-i)}_{k_{i}}: \left( z^{(-i)}_{k_{i}}, z^{(-(i - 1))}_{k_{i - 1}} \right) \in \Gamma_{\beta_{i}} \right\}$, respectively. The underline below the point $z_{0}$ is used to refer to the zeroth position in this part of the bi-infinite lettered word, while the vertical line `$\vert$' between $\beta_{1}$ and $\alpha_{1}$ refers to the zeroth position in this part of the bi-infinite lettered word. It is a simple observation that for any generic point $z_{0} \in \widehat{\mathbb{C}}$, we have $\displaystyle{\sum\limits_{\alpha\, =\, 1}^{N} m_{\alpha} \lambda_{\alpha} \left( z_{0} \right) = d_{{\rm fwd}}}$ and $\displaystyle{\sum\limits_{\beta\, =\, 1}^{N} m_{\beta} \delta_{\beta} \left( z_{0} \right) = d_{{\rm top}}}$. 

For any $E \subseteq \widehat{\mathbb{C}}$, we define $\displaystyle{\mathscr{O}^{\Gamma} \left(E \right)}$ to be $\displaystyle{\bigcup_{z_{0}\, \in\, E} \mathscr{O}^{\Gamma} \left( z_{0} \right)}$. In particular, $\displaystyle{\mathscr{O}^{\Gamma} \left( \widehat{\mathbb{C}} \right) = \bigcup_{z_{0}\, \in\, \widehat{\mathbb{C}}} \mathscr{O}^{\Gamma} \left( z_{0} \right)}$. Any point in $\mathscr{O}^{\Gamma} \left( \widehat{\mathbb{C}} \right)$ is denoted by $\mathfrak{X} \left( z_{0}; \boldsymbol{\gamma} \right)_{\boldsymbol{l}}$ where $z_{0} \in \widehat{\mathbb{C}},\ \boldsymbol{\gamma} = \left( \cdots \beta_{2}, \beta_{1} \vert \alpha_{1}, \alpha_{2}, \cdots  \right) \in \left\{ 1, 2, \cdots, N \right\}^{\mathbb{Z}}$ and $\boldsymbol{l} = \left( \cdots k_{2}, k_{1}, 0, j_{1}, j_{2}, \cdots \right)$, as described in Equation \eqref{eqn:di_gamma}. For any $r\, \in\, \mathbb{Z}$, let $\Pi_{r} : \mathscr{O}^{\Gamma} \left( \widehat{\mathbb{C}} \right) \longrightarrow \widehat{\mathbb{C}}$ denote the family of projections given by
\[ \Pi_{r} \left( \mathfrak{X} \left( z_{0}; \boldsymbol{\gamma} \right)_{\boldsymbol{l}} \right)\ \ =\ \ \begin{cases} z^{(r)}_{j_{r}} & \text{if}\ r \ge 0, \\ \vspace{-10pt} \\ z^{(r)}_{k_{-r}} & \text{if}\ r < 0. \end{cases} \]

Consider the shift map $\sigma^{\Gamma} : \mathscr{O}^{\Gamma} \left( \widehat{\mathbb{C}} \right) \longrightarrow \mathscr{O}^{\Gamma} \left( \widehat{\mathbb{C}} \right)$ given by 
\begin{eqnarray} 
\label{eqn:sigmadigamma} 
\sigma^{\Gamma} \left( \mathfrak{X} \left( z_{0}; \boldsymbol{\gamma} \right)_{\boldsymbol{l}} \right) & = & \sigma^{\Gamma} \left( \left( \cdots, z_{k_{2}}^{(-2)}, z_{k_{1}}^{(-1)}, \underline{z_{0}}, z_{j_{1}}^{(1)}, z_{j_{2}}^{(2)}, \cdots;\; \cdots, \beta_{2}, \beta_{1} \vert \alpha_{1}, \alpha_{2}, \cdots \right) \right) \nonumber \\ 
& = & \left( \cdots, z_{k_{1}}^{(-1)}, z_{0}, \underline{z_{j_{1}}^{(1)}}, z_{j_{2}}^{(2)}, z_{j_{3}}^{(3)}, \cdots;\; \cdots, \beta_{1}, \alpha_{1} \vert \alpha_{2}, \alpha_{3}, \cdots \right) \nonumber \\ 
& = & \mathfrak{X} \left( z_{j_{1}}^{(1)}; \sigma \boldsymbol{\gamma} \right)_{\sigma \boldsymbol{l}}, 
\end{eqnarray} 
whereby prefixing $\sigma$ to the bi-infinite lettered word $\boldsymbol{\gamma}$ or $\boldsymbol{l}$, we mean the bi-infinite lettered word that we obtain from $\boldsymbol{\gamma}$ or $\boldsymbol{l}$, as appropriate, by shifting it one place to the left. Also, we use the notation $\sigma^{m} \boldsymbol{\gamma}$ and $\sigma^{m} \boldsymbol{l}$ to denote $\sigma \left( \sigma^{m - 1} \boldsymbol{\gamma} \right)$ and $\sigma \left( \sigma^{m - 1} \boldsymbol{l} \right)$ respectively, for any $m \in \mathbb{Z}_{+}$, so that 
\[ \left( \sigma^{\Gamma} \right)^{m} \left( \mathfrak{X} \left( z_{0}; \boldsymbol{\gamma} \right)_{\boldsymbol{l}} \right)\ \ =\ \ \mathfrak{X} \left( z_{j_{m}}^{(m)}; \sigma^{m} \boldsymbol{\gamma} \right)_{\sigma^{m} \boldsymbol{l}}. \] 

Observe that the transformation $\sigma^{\Gamma}$ is invertible on the space $\mathscr{O}^{\Gamma} \left( \widehat{\mathbb{C}} \right)$ and the map $\left( \sigma^{\Gamma} \right)^{-1}$ is given by the right shift map defined analogously so that, for any $m \in \mathbb{Z}_{+}$, we have 
\[ \left( \sigma^{\Gamma} \right)^{-m} \left( \mathfrak{X} \left( z_{0}; \boldsymbol{\gamma} \right)_{\boldsymbol{l}} \right)\ \ =\ \ \mathfrak{X} \left( z_{k_{m}}^{(-m)}; \sigma^{-m} \boldsymbol{\gamma} \right)_{\sigma^{-m} \boldsymbol{l}}. \] 

One can observe that $\mathscr{O}^{\Gamma} \left( \widehat{\mathbb{C}} \right)$ is a metric space when equipped with the metric 
\begin{eqnarray*} 
& & d \left( \left( \mathfrak{X} \left( z_{0}; \boldsymbol{\gamma} \right)_{\boldsymbol{l}} \right),\; \left( \mathfrak{X} \left( w_{0}; \boldsymbol{\gamma'} \right)_{\boldsymbol{l'}} \right) \right) \\ 
& = & \max \Bigg\{ \sup_{n\, \in\, \mathbb{Z}} \left\{ \frac{1}{2^{|n|}} \rho \left( \Pi_{n} \left( \mathfrak{X} \left( z_{0}; \boldsymbol{\gamma} \right)_{\boldsymbol{l}} \right),\, \Pi_{n} \left( \mathfrak{X} \left( w_{0}; \boldsymbol{\gamma'} \right)_{\boldsymbol{l'}} \right) \right) \right\}, \\ 
& & \hspace{+4cm} \sup_{n\, \in\, \mathbb{Z_{+}}} \left\{ \frac{1}{2^{n}} \max \left\{ \left(1 - \delta_{\left( \alpha_{n},\, \alpha_{n}' \right)} \right), \left( 1 - \delta_{\left( \beta_{n},\, \beta_{n}'\right)}\right) \right\} \right\} \Bigg\}, 
\end{eqnarray*}
where $\boldsymbol{\gamma} = \left( \cdots \beta_{2}, \beta_{1} \vert \alpha_{1}, \alpha_{2}, \cdots  \right),\ \boldsymbol{\gamma'} = \left( \cdots \beta_{2}', \beta_{1}' \vert \alpha_{1}', \alpha_{2}', \cdots  \right),\ \rho$ is the spherical metric defined on $\widehat{\mathbb{C}}$ and $\delta_{\left( p, q \right)}$ is the Kronecker delta function. In the topology induced by the metric, readers may note that $\mathscr{O}^{\Gamma} \left( \widehat{\mathbb{C}} \right)$ is a compact metric space where $\sigma^{\Gamma}$ acts as a homeomorphism. 

We consider a variation of the following setting from \cite{sub:2026}. Let $\mathscr{M}_{\sigma^{\Gamma}} \left( \mathscr{O}^{\Gamma} \left( \widehat{\mathbb{C}} \right) \right)$ denote the collection of all Borel probability measures supported on $\mathscr{O}^{\Gamma} \left( \widehat{\mathbb{C}} \right)$ that remain $\sigma^{\Gamma}$-invariant. For any measure $\mathcal{M} \in \mathscr{M}_{\sigma^{\Gamma}} \left( \mathscr{O}^{\Gamma} \left( \widehat{\mathbb{C}} \right) \right)$, let $\left( \Pi_{r} \right)_{*} \mathcal{M}$ denote the push-forward measure, under the projection $\Pi_{r}$ for any $r \in \mathbb{Z}$ given by, 
\[ \left( \left( \Pi_{r} \right)_{*} \mathcal{M} \right) (B)\ \ =\ \ \mathcal{M} \left( \Pi_{r}^{-1} B \right),\ \ \ \text{for any Borel set}\ B \subset \widehat{\mathbb{C}}. \] 
This entails that for any $r \in \mathbb{Z}$ and for an arbitrary choice of a function $f \in \mathcal{C} \left( \widehat{\mathbb{C}}, \mathbb{R} \right)$, we have $\displaystyle{\int_{\widehat{\mathbb{C}}} f \mathrm{d}\left( \Pi_{r} \right)_{*} \mathcal{M} = \int_{\widehat{\mathbb{C}}} f \mathrm{d}\left( \Pi_{r + 1} \right)_{*} \mathcal{M}}$. Thus, $\left( \Pi_{r} \right)_{*} \mathcal{M} = \left( \Pi_{r + 1} \right)_{*} \mathcal{M}$ on $\widehat{\mathbb{C}}$ and hence, it is sufficient to work with the measure $\left( \Pi_{0} \right)_{*} \mathcal{M}$. However, note that there is a possibility that $\left( \Pi_{0} \right)_{*} \mathcal{M}_{1} \equiv \left( \Pi_{0} \right)_{*} \mathcal{M}_{2}$ on $\widehat{\mathbb{C}}$, even when $\mathcal{M}_{1} \ne \mathcal{M}_{2}$ on $\mathscr{O}^{\Gamma} \left( \widehat{\mathbb{C}} \right)$. Since $\mathscr{O}^{\Gamma} \left( \widehat{\mathbb{C}} \right)$ is a compact metric space and $\sigma^{\Gamma}$ is a continuous map there, we know from an elementary result in \cite{PolYuri:1998}, that $\mathscr{M}_{\sigma^{\Gamma}} \left( \mathscr{O}^{\Gamma} \left( \widehat{\mathbb{C}} \right) \right)$ contains at least one $\sigma^{\Gamma}$-ergodic measure, say $\mathcal{M}$ as an extreme point. We shall be interested in the push-forwards of such $\sigma^{\Gamma}$-ergodic measures from $\mathscr{M}_{\sigma^{\Gamma}} \left( \mathscr{O}^{\Gamma} \left( \widehat{\mathbb{C}} \right) \right)$ given by 
\begin{equation} 
\label{eqn:ergviasig} 
\mathscr{E}^{\Gamma}\ \ =\ \ \left\{ \mu \in \mathscr{M}(\widehat{\mathbb{C}})\; :\; \mu = \left( \Pi_{0} \right)_{*} \mathcal{M}\ \text{for some}\ \sigma^{\Gamma}\text{-ergodic}\ \mathcal{M} \in \mathscr{M}_{\sigma^{\Gamma}} \left( \mathscr{O}^{\Gamma} \left( \widehat{\mathbb{C}} \right) \right) \right\}. 
\end{equation} 

\section{Recurrence and ergodicity}
\label{sec:recc}

We begin by reviewing the standard notions and fundamental theorems concerning recurrence and ergodicity as established in \cite{lon:2022, lon:2024}, thus providing the necessary background for our study. 

\begin{definition} 
\label{defn:recc}
Let $\digamma$ be a holomorphic correspondence defined on $\widehat{\mathbb{C}}$. Suppose $E \subseteq \widehat{\mathbb{C}}$. A point $z_{0} \in E$ is said to be \emph{forward recurrent} on $E$ with respect to $\digamma$, if there exists $p \in \mathbb{Z}_{+}$ such that $\digamma^{p} \left( z_{0} \right) \cap E \ne \emptyset$. The minimum value of such $p \in \mathbb{Z}_{+}$ is defined to be the \emph{first forward return time} of $z_{0}$ in $E$, denoted by $\mathfrak{n}_{E} \left( z_{0} \right)$.  
\end{definition} 

\begin{remark} 
\label{defn:returntymback}
Observe that for any point $z_{0} \in E$, the first backward return time of $z_{0}$ in $E$ with respect to the holomorphic correspondence $\digamma$ denoted by $\mathfrak{n}^{\dagger}_{E} \left( z_{0} \right)$ is nothing but the first forward return time of $z_{0}$ in $E$ with respect to the holomorphic correspondence $\digamma^{\dagger}$.
\end{remark} 

Having defined the phenomenon of recurrence we now state the analogue of the Poincar\'{e} recurrence theorem for holomorphic correspondences. 

\begin{theorem}\cite{lon:2022} 
\label{thm:poinc}
Let $\digamma$ be a holomorphic correspondence defined on $\widehat{\mathbb{C}}$ and $\mu \in \mathscr{M} \left( \widehat{\mathbb{C}} \right)$ be a $\digamma^{*}$-invariant measure that does not put any mass on polar sets. Let $E \subseteq \widehat{\mathbb{C}}$ be such that $\mu (E) > 0$. Then, $\mu$-almost every point $z_{0} \in E$ is forward recurrent on $E$. 
\end{theorem} 

The concept of set invariance plays a vital role in the dynamical systems of maps and in determining whether the system is ergodic. However, since we are dealing with correspondences, the best we can ask for is something called almost invariance, which in turn, helps us in defining ergodicity. 

\begin{definition} 
\label{defn:Gammaai} 
Let $\digamma$ be a holomorphic correspondence defined on $\widehat{\mathbb{C}}$ and $\mu \in \mathscr{M} \left( \widehat{\mathbb{C}} \right)$ be a $\digamma^{*}$-invariant measure. A Borel set $E \subseteq \widehat{\mathbb{C}}$ is said to be \emph{almost invariant with respect to $\left( \digamma, \mu \right)$} if there exists a Borel set, say $E' \subseteq E$ such that $\mu (E') = \mu (E)$ and $\digamma^{\dagger} (E') \subseteq E$. 
\end{definition} 

Note that the support of the Dinh-Sibony measure denoted by $\Omega_{{\rm DS}}$, as defined in Example \ref{ex:invmeas}, is a trivial example of an almost invariant set with respect to $\left( \digamma, \omega_{{\rm DS}} \right)$. 

\begin{remark} 
\label{rmk:FGammak} 
Let $\digamma$ be a holomorphic correspondence defined on $\widehat{\mathbb{C}}$ and $\mu \in \mathscr{M} \left( \widehat{\mathbb{C}} \right)$ be a $\digamma^{*}$-invariant measure. Suppose $E$ is a Borel set that is almost invariant with respect to $\left( \digamma, \mu \right)$. Then, from \cite{lon:2024} we have,
\begin{enumerate} 
\item $E$ is almost invariant with respect to $\left( \digamma^{k}, \mu \right)$ for every $k \in \mathbb{Z}_{+}$. 
\item $E^{{\rm c}} = \widehat{\mathbb{C}} \setminus E$ is almost invariant with respect to $\left( \digamma, \mu \right)$. 
\end{enumerate} 
\end{remark} 

Using this notion of almost invariance of a set with respect to $\left( \digamma, \mu \right)$, Londhe defines the concept of ergodicity of a measure with respect to the holomorphic correspondence as such. 

\begin{definition} 
\label{defn:Gammaerg}
Let $\digamma$ be a holomorphic correspondence defined on $\widehat{\mathbb{C}}$. A $\digamma^{*}$-invariant measure $\mu \in \mathscr{M} \left( \widehat{\mathbb{C}} \right)$  is said to be \emph{ergodic} with respect to $\digamma$  if for every almost invariant Borel set $E \subseteq \widehat{\mathbb{C}}$ with respect to $(\digamma, \mu)$, we have $\mu (E) = 0$ or $1$. 
\end{definition} 

\section{Main results}
\label{sec:results}

In this section, we state the main results of this manuscript. Our first result concerns a quantification of the average of first return times of points in $E$.

\begin{theorem} 
\label{thm:kacforc} 
Let $\digamma$ be a holomorphic correspondence defined on $\widehat{\mathbb{C}}$. Suppose $\mu \in \mathscr{E}^{\Gamma}$ is a $\digamma^{*}$- invariant measure which does not put any mass on polar sets. Let $E \subseteq \widehat{\mathbb{C}}$ be such that $\mu (E) > 0$. Then, the average of the first return times satisfies 
\[ \int_{E} \mathfrak{n}_{E} (z)\, \mathrm{d}\mu_{E} (z)\ \ \le\ \ \frac{1}{\mu (E)},\ \ \ \ \text{where}\ \ \mu_{E} ( E' ) = \dfrac{\mu \left( E \cap E' \right)}{\mu (E)}\ \text{for any Borel subset}\  E' \subseteq \widehat{\mathbb{C}}.\]  
\end{theorem} 

The statements of Theorem \ref{thm:poinc} and Theorem \ref{thm:kacforc} pave the way for us to define another set-valued map, which one may call the induced correspondence. 

\begin{definition} 
\label{defn:indcorr} 
Let $\digamma$ be a holomorphic correspondence defined on $\widehat{\mathbb{C}}$ and $\mu \in \mathscr{M} \left( \widehat{\mathbb{C}} \right)$ be a $\digamma^{*}$-invariant measure. Let $E \subseteq \widehat{\mathbb{C}}$ be such that $\mu (E) > 0$. Then, for $\mu$-almost every $z \in E$, define a set-valued map $\digamma_{\hspace{-4pt} E}$ induced from $\digamma$ associated to $E$ as 
\begin{equation} 
\label{eqn:indcorr} 
\digamma_{\hspace{-4pt} E} (z)\ \ =\ \ \digamma^{\mathfrak{n}_{E} (z)} (z) \cap E, 
\end{equation} 
where $\mathfrak{n}_{E} (z)$ is the minimal positive integer, assured by Theorem \ref{thm:poinc}, for $\mu$-almost every $z \in E$. This set-valued map $\digamma_{\hspace{-4pt} E} : E \longrightarrow E$ is called \emph{an induced correspondence}.
\end{definition} 

In the dynamics of maps, the property of ergodicity is inherited by the induced system, whenever the underlying measure preserving transformation $T$ is ergodic. It is only natural to expect that the induced correspondence $\digamma_{\hspace{-4pt} E}$ is also ergodic, whenever the holomorphic correspondence $\digamma$ is ergodic. The following theorem precisely says this. 

\begin{theorem} 
\label{thm:erg}
Suppose $\mu$ is ergodic with respect to the holomorphic correspondence $\digamma$. Let $E \subseteq \widehat{\mathbb{C}}$ be such that $\mu (E) > 0$. Then, $\mu_{E}$ is ergodic with respect to the induced correspondence $\digamma_{\hspace{-4pt} E}$ as defined in Equation \eqref{eqn:indcorr}, meaning for every almost invariant Borel set $E' \subseteq E$ with respect to $\left( \digamma_{\hspace{-4pt} E}, \mu_{E} \right)$, we have $\mu_{E} (E') = 0$ or $1$. 
\end{theorem} 

We now introduce the \emph{Koopman operator} pertaining to the holomorphic correspondence  $\digamma$, denoted by $\mathcal{T}_{\digamma}$ defined on $\mathscr{L}^{1} \left( \widehat{\mathbb{C}}, \mu \right)$, whose action on the points in $\widehat{\mathbb{C}}$ is given by 
\begin{equation} 
\label{eqn:koopman} 
\left( \mathcal{T}_{\digamma} f \right) (z)\ \ =\ \ \frac{1}{d_{{\rm top}}} \sum_{w\, \in\, \digamma^{\dagger} (z)} f(w). 
\end{equation} 
It is a simple observation that $\mathcal{T}_{\digamma}$ is a bounded, linear operator with $\left\| \mathcal{T}_{\digamma} \right\| = 1$. 

The following theorem gives some important characterisations of ergodicity of holomorphic correspondences, measure-theoretically and using the Koopman operator. 

\begin{theorem} 
\label{thm:equierg}
Let $\mathcal{T}_{\digamma}$ be the Koopman operator pertaining to the holomorphic correspondence  $\digamma$, as defined in the Equation \eqref{eqn:koopman}. Suppose $\mu \in \mathscr{M} \left( \widehat{\mathbb{C}} \right)$ is a $\digamma^{*}$-invariant measure. Then the following statements are equivalent. 
\begin{enumerate} 
\item $\mu$ is ergodic with respect to the holomorphic correspondence $\digamma$. 
\item Every Borel set $E$ of strict positive measure $\mu$ satisfies $\displaystyle{\mu \left( \bigcup_{n\, \ge\, 1} \left( \digamma^{\dagger} \right)^{n} (E) \right) = 1}$. 
\item Suppose $E$ and $E'$ are Borel sets of strict positive measure $\mu$, then there exists $n \in \mathbb{Z}_{+}$ such that $\displaystyle{\mu \left( \left( \digamma^{\dagger} \right)^{n} (E) \cap E' \right) > 0}$. 
\item Suppose $f \in \mathscr{L}^{1} \left( \widehat{\mathbb{C}}, \mu \right)$ satisfies $\mathcal{T}_{\digamma} f \ge f,\ \mu$-almost everywhere, then $f$ is a constant function, $\mu$-almost everywhere. 
\item For $1 \le p < \infty$, suppose $f \in \mathscr{L}^{p} \left( \widehat{\mathbb{C}}, \mu \right)$ satisfies $\mathcal{T}_{\digamma} f = f$, $\mu$-almost everywhere, then $f$ is a constant function, $\mu$-almost everywhere. 
\end{enumerate} 
\end{theorem} 

While statements (2) and (3) of Theorem \ref{thm:equierg} concern backward orbits, the following result addresses the behavior of forward orbits under an ergodic correspondence.

\begin{theorem}
\label{thm:dense}
Let $\digamma$ be a holomorphic correspondence defined on $\widehat{\mathbb{C}}$ and $\mu \in \mathscr{M} \left( \widehat{\mathbb{C}} \right)$ be an ergodic measure with respect to $\digamma$ such that $\mu(B) > 0$ for any non-empty open set $B \subseteq \widehat{\mathbb{C}}$. Then, for $\mu$-almost every point in $\widehat{\mathbb{C}}$, the forward images under $\digamma$ is dense in $\widehat{\mathbb{C}}$. 
\end{theorem}

Finally, we conclude this section by stating an analogue of the Rokhlin-Kakutani lemma wherein we construct a weaker version of a Rokhlin tower of height $n$ with base $E$ and a residual set of measure atmost $\epsilon$.

\begin{theorem} 
\label{thm:rkforc}
Let $\digamma$ be a holomorphic correspondence defined on $\widehat{\mathbb{C}}$. Suppose $\mu = \left( \Pi_{0} \right)_{*}\mathcal{M}$ for some non-atomic, ergodic $\mathcal{M} \in \mathscr{M}_{\sigma^{\Gamma}} \left( \mathscr{O}^{\Gamma} \left( \widehat{\mathbb{C}} \right) \right)$. Then, for any $n \in \mathbb{Z}_{+}$ and $\epsilon > 0$, there exists a Borel set $E \in \mathcal{B}_{\mu}$ such that $\overline{\mu} \left( E \cup \digamma (E) \cup \cdots \cup \digamma^{n - 1} (E) \right) > 1- \epsilon$ where $\mathcal{B}_{\mu}$ denotes the completion of Borel $\sigma$-algebra $\mathcal{B}$ with respect to the measure $\mu$ and $\overline{\mu}$ denotes the completion of the measure $\mu$.
\end{theorem} 

The analysis developed in this work, except Theorem \ref{thm:dense}, extends to the setting of an arbitrary compact complex manifold $X$. Since Theorem \ref{thm:dense} uses the compact metric structure of $\widehat{\mathbb{C}}$ an extension to compact K\"{a}hler manifolds is possible in this case. The restriction to the Riemann sphere $\widehat{\mathbb{C}}$ is adopted, primarily for notational simplicity and expository convenience.

\section{Proof of Theorem \ref{thm:kacforc}}
\label{sec:kac_proof}

Our version of the Kac's theorem, as stated in Theorem \ref{thm:kacforc} draws its motivation from the Kac's theorem for maps, as stated in Theorem \ref{thm:kacfort}. 

\begin{proof}[of Theorem \ref{thm:kacforc}] 
Let $\mu$ be a measure satisfying the hypothesis of Theorem \ref{thm:kacforc} and $E \subseteq \widehat{\mathbb{C}}$ such that $\mu (E) > 0$. In order to prove $\displaystyle{\int_{E} \mathfrak{n}_{E} (z)\, \mathrm{d}\mu_{E} (z) \le \frac{1}{\mu (E)}}$, it is sufficient to prove $\displaystyle{\int_{E} \mathfrak{n}_{E} (z)\, \mathrm{d}\mu (z) \le 1}$. 

Since $\mu \in \mathscr{E}^{\Gamma}$, by definition we have $\mu = \left( \Pi_{0} \right)_{*} \mathcal{M}$ for some $\sigma^{\Gamma}$-ergodic measure $\mathcal{M} \in \mathscr{M}_{\sigma^{\Gamma}} \left( \mathscr{O}^{\Gamma} \left( \widehat{\mathbb{C}} \right) \right)$. Further, since $\mathscr{O}^{\Gamma} (E) = \left( \Pi_{0}\right)^{-1} (E) \subseteq \mathscr{O}^{\Gamma} (\widehat{\mathbb{C}})$, our hypothesis $\mu (E) > 0$ implies $\mathcal{M} \left( \mathscr{O}^{\Gamma} (E) \right) > 0$. Consequently, Theorem \ref{thm:kacfort} is applicable here. The vital step in this proof concerns the comparison of $\mathfrak{n}_{E} (z)$, as written in Definition \ref{defn:recc} for the set-valued map $\digamma$ and $n_{\left( \Pi_{0} \right)^{-1} E} \left( \mathfrak{X} \left( z; \boldsymbol{\gamma} \right)_{\boldsymbol{l}} \right)$ as defined in Equation \eqref{eqn:1}, for the shift map $\sigma^{\Gamma}$ that gives 
\[ 
\mathfrak{n}_{E}(z) \ \ = \ \mathfrak{n}_{E} \left( \Pi_{0} \left( \mathfrak{X} \left( z; \boldsymbol{\gamma} \right)_{\boldsymbol{l}} \right) \right)\ \ \le\ \ n_{\left( \Pi_{0} \right)^{-1} E} \left( \mathfrak{X} \left( z; \boldsymbol{\gamma} \right)_{\boldsymbol{l}} \right)\ \ \ \ \forall \  \mathfrak{X} \left( z; \boldsymbol{\gamma} \right)_{\boldsymbol{l}} \in \left( \Pi_{0} \right)^{-1} E. 
\]
Thus, 
\begin{eqnarray*} 
\int_{E} \mathfrak{n}_{E} (z) \mathrm{d}\mu (z) & = & \int_{\mathscr{O}^{\Gamma} (E)} \left( \mathfrak{n}_{E} \circ \Pi_{0}\right) \left( \mathfrak{X} \left( z; \boldsymbol{\gamma} \right)_{\boldsymbol{l}} \right) \mathrm{d}\mathcal{M} \left( \mathfrak{X} \left( z; \boldsymbol{\gamma} \right)_{\boldsymbol{l}} \right) \\ 
& \le & \int_{\left( \Pi_{0} \right)^{-1} E} n_{\left( \Pi_{0} \right)^{-1} E} \left( \mathfrak{X} \left( z; \boldsymbol{\gamma} \right)_{\boldsymbol{l}} \right) \mathrm{d}\mathcal{M} \left( \mathfrak{X} \left( z; \boldsymbol{\gamma} \right)_{\boldsymbol{l}} \right) \\ 
& = & 1 \hspace{7cm} \text{(by Theorem \ref{thm:kacfort})}, 
\end{eqnarray*} 
thereby, completing the proof. 
\end{proof} 

Using a similar argument, we can prove the following version of Theorem \ref{thm:kacforc} for the backward return time $\mathfrak{n}^{\dagger}_{E} (z)$, as described in Remark \ref{defn:returntymback}. The proof of the same will follow \emph{mutatis mutandis} to the proof of Theorem \ref{thm:kacforc}.

\begin{corollary}
\label{cor:kac}
Let $\digamma$ be a holomorphic correspondence defined on $\widehat{\mathbb{C}}$. Suppose $\nu \in \mathscr{E}^{\Gamma}$ is a $\digamma^{*}$-invariant measure. Let $E \subseteq \widehat{\mathbb{C}}$ be such that $\nu (E) > 0$. Then, the average of the first return time satisfies 
\[ \int_{E} \mathfrak{n}^{\dagger}_{E} (z)\, \mathrm{d}\nu_{E} (z)\ \ \le\ \ \frac{1}{\nu (E)},\ \ \ \ \text{where}\ \ \nu_{E} ( E' ) = \dfrac{\nu \left( E \cap E' \right)}{\nu (E)} \ \ \ \text{for any Borel subset}\ E' \subseteq \widehat{\mathbb{C}}. \] 
\end{corollary}

We now make a brief observation concerning the ergodic measures for invertible measure preserving transformations. Consider $(X, \mathscr{B}, M, T)$ to be an invertible measure-preserving system. Then, it is an elementary observation that the measure $M$ is ergodic with respect to $T$ if and only if it is ergodic with respect to $T^{-1}$. When applied to the space $\mathscr{O}^{\Gamma}\left(\widehat{\mathbb{C}}\right)$, this yields another representation for any measure $\mu \in \mathscr{E}^{\Gamma}$ given by $(\Pi_{0})_{*}\mathcal{M}$, where $\mathcal{M}$ is ergodic with respect to $(\sigma^{\Gamma})^{-1}$. Readers interested in the proof of Corollary \ref{cor:kac} may need to use the above-mentioned idea, while we will invoke this implicitly in the proof of Proposition \ref{prop:erg}.

\subsection{Special cases and Examples}
\label{spcl cases}
Careful readers might have observed that Theorem \ref{thm:kacforc} and Corollary \ref{cor:kac} are achieved with a push-forward of an ergodic measure $\mathcal{M} \in \mathscr{M}_{\sigma^{\Gamma}} \left( \mathscr{O}^{\Gamma} \left( \widehat{\mathbb{C}} \right) \right)$. Here, we look at some interesting cases of the analogue of Kac's theorem when the concerned measure is ergodic with respect to $\digamma$. 

\begin{enumerate} 
\item[{\bf Case 1:}]  
Let $\digamma$ be a holomorphic correspondence defined on $\widehat{\mathbb{C}}$ and $\mu$ be an ergodic measure with respect to the holomorphic correspondence $\digamma$. Let $E \subseteq \widehat{\mathbb{C}}$ be an almost invariant set with respect to $\left( \digamma, \mu \right)$ such that $\mu (E) > 0$. Then, $\mu (E) = 1$. Further, Definition \ref{defn:Gammaai} gives $\displaystyle{\mathfrak{n}_{E}^{\dagger} (z) = 1}$, for $\mu$-almost every $z \in E$. Hence, $\displaystyle{\int_{E} \mathfrak{n}_{E}^{\dagger} (z) \mathrm{d}\mu (z) = 1}$. 
\end{enumerate} 

We next examine the special case where $E = \Omega_{{\rm DS}}$. Although this is a special case of the above, we treat it separately due to the independent dynamical significance of $\Omega_{{\rm DS}}$.

\begin{enumerate} 
\item[{\bf Case 2:}] 
Suppose $E = \Omega_{{\rm DS}}$ in Case 1 and the ergodic measure $\mu$ satisfies $\mu \left( \Omega_{{\rm DS}} \right) > 0$. Then, $\mu \left( \Omega_{{\rm DS}} \right) = 1$. Moreover, we know from Example \ref{ex:invmeas} that $\displaystyle{\digamma \left( z \right) \cap \Omega_{{\rm DS}} \ne \emptyset\ \forall z \in \Omega_{{\rm DS}}}$. Thus, $\mathfrak{n}_{\Omega_{{\rm DS}}}(z) = 1\ \forall z \in \Omega_{{\rm DS}}$. Hence, 
\begin{eqnarray}
\label{eqn:2}
\int_{\Omega_{{\rm DS}}} \mathfrak{n}_{\Omega_{{\rm DS}}} (z) \mathrm{d}\mu (z)\ \ =\ \ 1\ \ =\ \ \int_{\Omega_{{\rm DS}}} \mathfrak{n}_{\Omega_{{\rm DS}}}^{\dagger} (z) \mathrm{d}\mu (z).
\end{eqnarray} 
\end{enumerate} 

The following explicit example illustrates both the cases discussed above. We consider a holomorphic correspondence which can be associated with a finitely generated rational semigroup. Interested readers may refer to \cite{hinka:1996, bs:2017} for more details. 

\begin{example} 
\label{ex:hinka}
Consider the holomorphic correspondence $\digamma_{\hspace{-4pt} S}$ associated with the finitely generated rational semigroup $\left\langle z^{2},\frac{z^{2}}{2} \right\rangle$. Consider the Dinh-Sibony measure $\omega_{{\rm DS}}$, associated with $\digamma_{\hspace{-4pt} S}$, as defined in Example \ref{ex:invmeas}. Boyd, in \cite{boyd:1999} came up with an explicit expression for the same given by $\omega_{{\rm DS}}(B) = \dfrac{\lambda (\log B)}{2 \pi \log 2}$, for any Borel set $B \subseteq \widehat{\mathbb{C}}$. Here, $\log B$ stands for the image of the Borel set $B$ with respect to the principal branch of the logarithmic map and $\lambda$ is the Lebesgue measure on $\mathbb{C}$. Then, the measure $\omega_{{\rm DS}}$ is ergodic with respect to $\digamma_{\hspace{-4pt} S}$. In this setting, $\Omega_{{\rm DS}} = \left\{ z \in \widehat{\mathbb{C}} : 1 \le |z| \le 2\right\}$ and consequently, Equation \eqref{eqn:2} holds true. 
\end{example} 

We now present an example, within the same framework as in the Example \ref{ex:hinka}, demonstrating that the equality can be attained for nontrivial subsets of $\Omega_{{\rm DS}}$.

\begin{example}
\label{ex:1}
Consider $\digamma_{\hspace{-4pt} S},\ \omega_{{\rm DS}}$ and $\Omega_{{\rm DS}}$ as given in the Example \ref{ex:hinka}. Let $E \subsetneq \Omega_{{\rm DS}}$ be the annulus $\displaystyle{E = \left\{ z \in \Omega_{{\rm DS}} : 2^{\frac{1}{4}} \le |z| \le 2^{\frac{3}{4}} \right\}}$. Note that $\omega_{{\rm DS}} (E) = \dfrac{1}{2}$. Partitioning the set $E$ as 
\begin{equation} 
\label{eqn:3} 
E\ \ =\ \ \bigcup_{k\, \in\, \mathbb{Z}_{+}} E_{k}\ \ \ \ \text{where}\ \ \ E_{k}\ \ =\ \ \bigg\{ z \in E\ :\ \mathfrak{n}_{E} (z) = k \bigg\}, 
\end{equation} 
one can verify by induction that for each $k \in \mathbb{Z}_{+}$, we have
\[ E_{k}\ \ =\ \ \bigg\{ z \in \widehat{\mathbb{C}}\; :\; 2^{\frac{2^{k + 1} - 2}{2^{k + 2}}} \le |z| \le 2^{\frac{2^{k + 1} - 1}{2^{k + 2}}} \bigg\}\ \bigsqcup\ \bigg\{ z \in \widehat{\mathbb{C}}\; :\; 2^{\frac{2^{k + 1} + 1}{2^{k + 2}}} \le |z| \le 2^{\frac{2^{k + 1} + 2}{2^{k + 2}}} \bigg\}, \] 
that, in turn gives $\displaystyle{\int_{E} \mathfrak{n}_{E} (z) \, \mathrm{d}\omega_{{\rm DS}} = \sum_{k\, \in\, \mathbb{Z}^{+}} \left( \int_{E_{k}} k \, \mathrm{d}\omega_{{\rm DS}} \right) = \sum_{k\, \in\, \mathbb{Z}^{+}} k \, \omega_{{\rm DS}} \left( E_{k} \right)}$. Using the expression of $\omega_{{\rm DS}}$, as given in the Example \ref{ex:hinka}, we obtain $\displaystyle{\int_{E} \mathfrak{n}_{E} (z)\, \mathrm{d}\omega_{{\rm DS}} = \sum_{k\, \in\, \mathbb{Z}_{+}} \frac{k}{2^{k + 1}} = 1}$.
\flushright\end{example}

The following example heuristically demonstrates why $\digamma^{*}$-invariance of the concerned measure is insufficient to yield an analogue of Kac's theorem for the system of holomorphic correspondences.

\begin{example}
\label{ex:2}
Consider the holomorphic correspondence $\digamma_{\hspace{-4pt} P}$ given by the zero set of a polynomial in two variables, say $\displaystyle{P \left( z, w \right) = \sum_{k_{1}\, =\, 0}^{d} \sum_{k_{2}\, =\, 0}^{d} C_{\left( k_{1}, k_{2} \right)} z^{k_{1}} w^{k_{2}}}$. Such a polynomial correspondence is said to be a \emph{symmetric separable polynomial correspondence} (see \cite{bks:2026, bks:corr:2026, gst:2023}) if  
\begin{enumerate} 
\item No linear polynomial in $z$ or $w$ divides $P$, {\it i.e.}, there exists no point $a \in \mathbb{C}$ such that \\ $(z - a) | P \left( z, w \right)$ or $(w - a) | P \left( z, w \right)$; 
\item The equation $P \left( z, w \right) = 0$ can be alternatively written as $R_{1} (z) R_{2} (w) = 1$, where $R_{k}$ is a rational map of degree $d$ in $z$ and $w$ respectively for $k = 1, 2$. 
\item $P \left( z, w \right) = 0$ if and only if $P \left( w, z \right) = 0$, in other words $C_{\left( k_{1}, k_{2} \right)} = C_{\left( k_{2}, k_{1} \right)}$ for all $0 \le k_{1}, k_{2} \le d$. 
\end{enumerate} 

For any set $E \subseteq \widehat{\mathbb{C}}$, note that $\mathfrak{n}_{E} (z) \le 2$ and $\mathfrak{n}_{E}^{\dagger} (z) \le 2$. Suppose $\mu$ is an ergodic measure with respect to $\digamma_{\hspace{-4pt} P}$. Choose $E \subseteq \widehat{\mathbb{C}}$ such that $0 < \mu (E) \le \dfrac{1}{2}$. Then, we have $\displaystyle{\int_{E} \mathfrak{n}_{E} (z) \mathrm{d}\mu \le 1}$. On the other hand, if $E \subseteq \widehat{\mathbb{C}}$ satisfies $\dfrac{1}{2} < \mu (E) \le 1$ then $\displaystyle{\int_{E} \mathfrak{n}_{E} (z) \mathrm{d}\mu > 1}$. An analogous conclusion holds for $\mathfrak{n}_{E}^{\dagger} (z)$. As an explicit example of the above class of polynomial correspondence, one can consider $P(z, w) = z^{2} - w^{2}$. Then, for any $E \subseteq \widehat{\mathbb{C}}$, we have $\mathfrak{n}_{E} (z) = \mathfrak{n}_{E}^{\dagger} (z) = 1$ for every $z \in E$.
\end{example}

Note that the Examples \ref{ex:hinka} and \ref{ex:2} give classes of holomorphic correspondences and Borel sets $E$ for which the forward return time is uniformly bounded. More importantly, Examples \ref{ex:1} and \ref{ex:2} demonstrate that, unlike in the case of maps, ergodicity of the concerned measure with respect to the correspondence $\digamma$ imposes no restriction on the value of the integral $\displaystyle{\int_{E} \mathfrak{n}_{E} (z) \mathrm{d}\mu (z)}$. 

A major distinction in the phenomenon of recurrence between the dynamical systems of maps and that of holomorphic correspondences can be described as follows. Let $(X, \mathscr{B}, M, T)$ be an ergodic measure preserving system and $E \in \mathscr{B}$ such that $M(E) > 0$. Suppose $n_{E} (x) = 1$ for $M$-almost every $x \in E$. Readers can verify that $M \left( E\, \triangle\, T^{-1} (E) \right) = 0$. An application of Theorem 1.8 in \cite{walt:1982} then yields $M (E) \in \{ 0, 1 \}$. Consequently $M(E)=1$, which reduces the situation into a trivial case. In contrast, for systems arising from holomorphic correspondences, such a phenomenon does not necessarily occur. In particular, for a set $E$ satisfying $0<\mu(E)<1$, it is possible that $\mathfrak{n}_{E} (z) = 1$ for $\mu$-almost every $z \in E$, as demonstrated by the explicit construction in Example \ref{ex:2}.

\section{The induced correspondence} 
\label{sec:induced} 

Throughout this section, we assume that $\mu$ is a $\digamma^{*}$-invariant measure and that $E \subseteq \widehat{\mathbb{C}}$ is a Borel set of strict positive measure. As established in Section \ref{sec:results}, the analogue of Poincaré recurrence holds in this setting, making it natural to define the first return time as in Definition \ref{defn:recc}. This paves the way, as in the case of dynamics of maps, for the definition of an induced correspondence, $\digamma_{\hspace{-4pt} E}$ as written in Definition \ref{defn:indcorr}. If $\mathscr{B}$ denotes the Borel $\sigma$-algebra of $\widehat{\mathbb{C}}$, then we know that $\displaystyle{\mathscr{B}_{E} = \left\{ B \cap E : B \in \mathscr{B} \right\}}$ defines an induced Borel $\sigma$-algebra on $E$. Similarly, if $\mu$ is a measure supported on $\widehat{\mathbb{C}}$, then $\mu_{E}$, as given in Theorem \ref{thm:kacforc}, defines an induced measure on $E$. The set $E$ equipped with the Borel $\sigma$-algebra $\mathscr{B}_{E}$ and the induced measure $\mu_{E}$, along with the set-valued map $\digamma_{\hspace{-4pt} E}$ forms the induced system. 

In case of maps whenever $(X, \mathscr{B}, M, T)$ is a measure preserving system, the induced system given by $\left( E,\mathscr{B}_{E}, M_{E}, T_{E} \right)$ is also measure preserving. However, in the case of correspondences, this is not so. An induced correspondence does not naturally acquire a variety representation, thereby making it merely a set-valued map constructed from an underlying holomorphic correspondence. This also has a few immediate consequences. In particular, one cannot, in general, define the topological degree, nor can one establish an invariance of measures for the induced correspondence $\digamma_{\hspace{-4pt} E}$.

To study the dynamics of $\digamma_{\hspace{-4pt} E}$, it is necessary to introduce the notions of inverse and composition, therein. The adjoint of the set-valued map stemming from the induced correspondence $\digamma_{\hspace{-4pt} E}$, is given by $\digamma_{\hspace{-4pt} E}^{\dagger} (z) = \left\{ w \in E : z \in \digamma^{\mathfrak{n}_{E}(w)} (w) \right\}$, for $\mu$-almost every $z \in E$. And for a set $B \subseteq E$, 
\[ \digamma_{\hspace{-4pt} E}^{\dagger} (B)\ \ =\ \ \left\{ z \in E\; :\; \digamma^{\mathfrak{n}_{E} (z)} (z) \cap B \ne \emptyset \right\}. \] 
The iterates of the induced correspondence is defined inductively thus: For $k \in \mathbb{Z}^{+}$, we set
\[ \digamma_{\hspace{-4pt} E}^{k} (z)\ \ =\ \ \digamma^{\mathfrak{n}_{E} \left( z' \right)} \left( z' \right) \cap E,\ \ \text{where}\ \ z' \in \digamma_{\hspace{-4pt} E}^{k - 1} (z). \] 

A simple instance of an induced system arises under the setting described in the Example \ref{ex:hinka}. In this situation, the set-valued map $\digamma_{\hspace{-4pt} E}$ acts as $\displaystyle{\digamma_{\hspace{-4pt} E} (z) = \digamma(z) \cap E}$, for $\omega_{{\rm DS}}$-almost every $z \in \Omega_{{\rm DS}}$, since $\mathfrak{n}_{E} (z) = 1$ for every $z \in \Omega_{{\rm DS}}$. We now describe a non-trivial example that demonstrates the role of forward return times in shaping the course of induced dynamics, using the settings of Example \ref{ex:1}.

\begin{example}
Consider $\digamma_{\hspace{-4pt} S},\ \omega_{{\rm DS}}$ and $\Omega_{{\rm DS}}$ as given in the Example \ref{ex:hinka}. Let $E \subset \Omega_{{\rm DS}}$ be the annulus $E = \left\{ z \in \Omega_{{\rm DS}} : 2^{1/4} \le |z| \le 2^{3/4} \right\}$. As in Example \ref{ex:2}, we partition the set $E$ according to the first return time into $E_{k}$'s, as follows. 
\[ E_{k}\ \ =\ \ \bigg\{ z \in \widehat{\mathbb{C}} : 2^{\frac{2^{k + 1} - 2}{2^{k + 2}}} \le |z| \le 2^{\frac{2^{k + 1} - 1}{2^{k + 2}}} \bigg\}\, \bigsqcup\, \bigg\{ z \in \widehat{\mathbb{C}} : 2^{\frac{2^{k + 1} + 1}{2^{k + 2}}} \le |z| \le 2^{\frac{2^{k + 1} + 2}{2^{k + 2}}} \bigg\}. \] 
Then the action of $\digamma_{\hspace{-4pt} E}$ is given by $\digamma_{\hspace{-4pt} E}(z) = \digamma_{\hspace{-4pt} S}^{k} (z) \cap E$, where $z \in E_{k}$. 
\end{example}

With the above notations in mind, we now enlist a few results regarding some measure theoretic properties of the induced system, before embarking on the proof of Theorem \ref{thm:erg}. 
\begin{property} 
The image of a Borel set under the induced correspondence $\digamma_{\hspace{-4pt} E}$ is a Borel set, {\it i.e.}, if $B \in \mathscr{B}_{E}$, then $\digamma_{\hspace{-4pt} E} \left( B \right) \in \mathscr{B}_{E}$. 
\end{property} 

\begin{proof} 
Consider a Borel set $B \in \mathscr{B}_{E}$. Suppose we partition the set $E$ as described in Equation \eqref{eqn:3}. Then $\displaystyle{B = \bigcup_{k\, \in\, \mathbb{Z}_{+}} \left( B \cap E_{k} \right)}$. Thus, $\displaystyle{\digamma_{\hspace{-4pt} E} (B) = \bigcup_{k\, \in\, \mathbb{Z}_{+}} \left( \digamma^{k} \left( B \cap E_{k} \right) \cap E \right)}$. Since $E_{k}$ can be given alternatively by $\displaystyle{E_{k} =  \left( E \cap \left( \digamma^{\dagger} \right)^{k} (E) \right) \setminus \left( \bigcup_{1\, \le\, j\, \le\, k - 1} E_{j} \right)}$, it is clear that every $E_{k}$ is a Borel set, and thus, so are $\displaystyle{\digamma^{k} \left( B \cap E_{k} \right)}$ for any $k \in \mathbb{Z}_{+}$, thereby making $\digamma_{\hspace{-4pt} E} (B)$ a Borel set. 
\end{proof}

The next property concerns the relationship between the invariance of the induced measure $\mu_{E}$ and the almost invariance of the set $E$ with respect to $\left( \digamma, \mu \right)$. Note that the second part of the following proposition has been mentioned in Corollary 5.4 in \cite{lon:2024}.

\begin{proposition}
\label{prop:iminv}
Let $\digamma$ be a holomorphic correspondence defined on $\widehat{\mathbb{C}}$ and $\mu$ be a $\digamma^{*}$-invariant measure supported on $\widehat{\mathbb{C}}$. Let $E \subseteq \widehat{\mathbb{C}}$, such that $\mu(E)>0$. Then the induced measure $\mu_{E}$ is $\digamma^{*}$-invariant if and only if $E$ is almost invariant with respect to $\left( \digamma, \mu \right)$.
\end{proposition} 

\begin{proof} 
Suppose $\mu_{E}$ is a $\digamma^{*}$-invariant measure, we have 
\begin{equation} 
\label{eqn:4} 
\frac{1}{d_{{\rm top}}} \int_{\widehat{\mathbb{C}}} \sum_{w\, \in\, \digamma^{\dagger} (z)} f(w) \mathrm{d}\mu_{E}(z)\ \ =\ \ \int_{\widehat{\mathbb{C}}} f(z) \mathrm{d}\mu_{E}(z)\ \ \ \ \forall f \in \mathcal{C} \left( \widehat{\mathbb{C}},\, \mathbb{R} \right). 
\end{equation} 

As mentioned in Section \ref{sec:prelims}, the above equality extends to $\mathscr{L}^{1}$-functions as well. Choosing $f \equiv \chi_{E}$, in Equation \eqref{eqn:4}, we get 
\[ \frac{1}{d_{{\rm top}}}\int_{\widehat{\mathbb{C}}} \sum_{w\, \in\, \digamma^{\dagger} (z)} \chi_{E} (w) \mathrm{d}\mu_{E} (z)\ \ =\ \ \int_{\widehat{\mathbb{C}}} \chi_{E} (z) \mathrm{d}\mu_{E}(z). \] 
The right hand side of the above equation is equal to $1$, while the left hand side, upon simplification, equals $\displaystyle{\frac{1}{\mu(E)} \frac{1}{d_{{\rm top}}}\int_{E} \sum_{w\, \in\, \digamma^{\dagger} (z)} \chi_{E} (w) \mathrm{d}\mu(z)}$. Consequently, we obtain 
\[ \bigintss_{E} \left( 1 - \frac{1}{d_{{\rm top}}} \sum_{w\, \in\, \digamma^{\dagger} (z)} \chi_{E} (w) \right) \mathrm{d}\mu (z)\ \ =\ \ 0. \] 
Since the integrand is a non-negative $\mathscr{L}^{1}$-function, we obtain $\displaystyle{\chi_{E} (z) = \frac{1}{d_{{\rm top}}} \sum_{w\, \in\, \digamma^{\dagger} (z)} \chi_{E} (w)}$ for $\mu$-almost every $z \in E$. This proves the almost invariance of the set $E$ with respect to $\left( \digamma, \mu \right)$.

In order to prove that $\mu_{E}$ is $\digamma^{*}$-invariant, we need to show that Equation \eqref{eqn:4} holds for every function $f \in \mathcal{C} \left( \widehat{\mathbb{C}}, \mathbb{R} \right)$. Towards that end, consider an arbitrary function $f \in \mathcal{C} \left( \widehat{\mathbb{C}}, \mathbb{R} \right)$ and note that 
\[ \int_{\widehat{\mathbb{C}}} \sum_{w\, \in\, \digamma^{\dagger} (z)} f(w) \mathrm{d}\mu_{E}(z)\ \ =\ \ \frac{1}{\mu(E)} \int_{E} \sum_{w\, \in\, \digamma^{\dagger} (z)} f(w) \mathrm{d}\mu(z). \] 
Further, making use of $\digamma^{*}$-invariance of $\mu$, the almost invariance of the set $E$ with respect to $\left( \digamma, \mu \right)$, and thereby the almost invariance of the set $E^{{\rm c}}$ with respect to $\left( \digamma, \mu \right)$ (see Remark \ref{rmk:FGammak}), we have
\[ \int_{\widehat{\mathbb{C}}} f(z) \mathrm{d}\mu_{E}(z)\ \ =\ \ \frac{1}{\mu(E)} \int_{\widehat{\mathbb{C}}} \left( f \cdot \chi_{E} \right) (z) \mathrm{d}\mu(z)\ \ =\ \ \frac{1}{d_{{\rm top}}} \frac{1}{\mu(E)} \int_{E} \sum_{w\, \in\, \digamma^{\dagger} (z)} \left( f \cdot \chi_{E} \right) (w) \mathrm{d}\mu(z). \] 

Equating the two expressions above establishes Equation \eqref{eqn:4}. Since this identity holds for an arbitrary $f \in \mathcal{C}\left(\widehat{\mathbb{C}},\, \mathbb{R}\right)$, the proof of the proposition is complete.

\end{proof} 

\subsection{Proof of Theorem \ref{thm:erg}}

First, we prove a lemma which discusses the almost invariance of subsets of $\widehat{\mathbb{C}}$ with respect to $\left( \digamma, \mu \right)$ and $\left( \digamma_{\hspace{-4pt} E}, \mu_{E} \right)$. The definition of the former is as written in Definition \ref{defn:Gammaai}, while the definition of the latter is as written below. 

\begin{definition} 
\label{defn:Lambdaai} 
Let $\digamma_{\hspace{-4pt} E}$ be the induced correspondence, as written in Definition \ref{defn:indcorr} and $\mu_{E}$ be the induced measure, as written in Theorem \ref{thm:kacforc}. We call a set $B \subseteq E$ to be \emph{almost invariant with respect to $\left( \digamma_{\hspace{-4pt} E}, \mu_{E} \right)$} if there exists a subset $B' \subseteq B$ such that $\digamma_{\hspace{-4pt} E}^{\dagger} (B') \subseteq B$ and $\mu_{E} (B') = \mu_{E} (B)$.
\end{definition} 

\begin{lemma}
\label{prop:usedintheorem} 
Suppose $B \subseteq \widehat{\mathbb{C}}$ is almost invariant with respect to $\left( \digamma, \mu \right)$, then $B \cap E$ is almost invariant with respect to $\left( \digamma_{\hspace{-4pt} E}, \mu_{E} \right)$. If $B \subseteq E$ is almost invariant with respect to $\left( \digamma_{\hspace{-4pt} E}, \mu_{E} \right)$, then $B$ is almost invariant with respect to $\left( \digamma, \mu \right)$. 
\end{lemma} 

\begin{proof} 
Let $B \subseteq \widehat{\mathbb{C}}$ be an almost invariant set with respect to $\left( \digamma, \mu \right)$. Then, by Remark \ref{rmk:FGammak}, we know that $B$ is an almost invariant set with respect to $\left( \digamma^{k}, \mu \right)$ for every $k \in \mathbb{Z}_{+}$. Thus, applying the construction due to Londhe (refer Lemma 3.1 in \cite{lon:2024}), we obtain a decreasing sequence of Borel sets $B_{k} \subset B$ such that $\displaystyle{\left( \digamma^{\dagger}\right)^{k} \left( B_{k} \right) \subseteq B}$ with $\mu \left( B_{k} \right) = \mu (B)$ for every $k \in \mathbb{Z}_{+}$. Define $\displaystyle{B_{*} = \bigcap_{k\, \in\, \mathbb{Z}_{+}} B_{k}}$. Then, for every $k \in \mathbb{Z}_{+}$, we have $\displaystyle{\left( \digamma^{\dagger}\right)^{k} \left( B_{*} \right) \subseteq B}$ and $\displaystyle{\mu \left( B_{*} \right) = \mu (B)}$.

To prove $B \cap E$ is almost invariant with respect to $\left( \digamma_{\hspace{-4pt} E}, \mu_{E} \right)$, we consider the subset $B_{*} \cap E \subseteq B \cap E$ and prove that $\digamma_{\hspace{-4pt} E}^{\dagger} \left( B_{*} \cap E \right) \subseteq B \cap E$ and $\mu \left( B_{*} \cap E \right) = \mu \left( B \cap E \right)$. By the definition of the set $B_{*}$, we have $\mu \left( B_{*} \cap E \right) \le \mu \left( B \cap E \right)$. Further, since $B_{*} \cup E \subseteq B \cup E$, we also have that $\mu \left( B_{*} \cap E \right) \ge \mu \left( B \cap E \right)$. 

Now making use of the sequence of sets $E_{k}$'s as defined in Equation \eqref{eqn:3} we have 
\begin{eqnarray} 
\digamma_{\hspace{-4pt} E}^{\dagger} \left( B_{*} \cap E \right) & = & \left\{ z \in \bigcup_{k\, \in\, \mathbb{Z}_{+}} E_{k} : \digamma^{\mathfrak{n}_{E} (z)} (z) \cap \left( B_{*} \cap E \right) \ne \emptyset \right\} \nonumber \\ 
\label{eqn:5} 
& = & \bigcup_{k\, \in\, \mathbb{Z}_{+}} \left( E_{k} \cap \left( \digamma^{\dagger} \right)^{k} \left( B_{*} \cap E \right) \right)\\   
&\subseteq & \bigcup_{k\, \in\, \mathbb{Z}_{+}} \left[ \left( E_{k} \cap \left( \digamma^{\dagger} \right)^{k} \left( B_{*} \right) \right) \cap \left( E_{k} \cap \left( \digamma^{\dagger} \right)^{k} \left( E \right) \right) \right]\nonumber \\ 
& \subseteq & \bigcup_{k\, \in\, \mathbb{Z}_{+}} \left[ \left( E_{k} \cap \left( \digamma^{\dagger} \right)^{k} \left( B_{*} \right) \right) \right] \cap E \nonumber \\ 
& \subseteq & B \cap E \nonumber. 
\end{eqnarray}

The proof of the second part of the Lemma \ref{prop:usedintheorem} is as follows. Since $B \subseteq E$ is almost invariant with respect to $\left( \digamma_{\hspace{-4pt} E}, \mu_{E} \right)$, we have by Definition \ref{defn:Lambdaai} that there exists a subset $B' \subseteq B$ such that $\digamma_{\hspace{-4pt} E}^{\dagger} (B') \subseteq B$ and $\mu_{E} (B') = \mu_{E} (B)$. We now prove that the same subset $B'$ helps us achieve the almost invariance of the set $B$, with respect to $\left( \digamma, \mu \right)$. Towards that end, note that similar to the Equation \eqref{eqn:5}, we have $\displaystyle{\digamma_{\hspace{-4pt} E}^{\dagger} (B') = \bigcup_{k\, \in\, \mathbb{Z}_{+}} \left( E_{k} \cap \left( \digamma^{\dagger} \right)^{k} (B') \right)}$. Now suppose $z \in \digamma^{\dagger} (B')$, then $\mathfrak{n}_{E} (z) = 1$, {\it i.e.}, $z \in E_{1}$. Thus, $\digamma^{\dagger} (B') \subseteq \digamma_{\hspace{-4pt} E}^{\dagger} (B') \subseteq B$. Also, $\mu_{E} \left( B' \right) = \mu_{E} \left( B \right)$ implies $\mu \left( B' \right) = \mu \left( B \right)$, since $B' \subseteq B \subseteq E$. 
\end{proof} 

Now, we are ready to prove Theorem \ref{thm:erg}, where we show that the induced measure $\mu_{E}$ is ergodic with respect to $\digamma_{\hspace{-4pt} E}$, assuming the ergodicity of $\mu$ with respect to $\digamma$.

\begin{proof}[of Theorem \ref{thm:erg}] 
Consider an almost invariant set $B \in \displaystyle{\mathscr{B}_{E}}$ with respect to $\left( \digamma_{\hspace{-4pt} E}, \mu_{E} \right)$. Then, by Lemma \ref{prop:usedintheorem}, the set $B$ is almost invariant with respect to $\left( \digamma, \mu \right)$. Making use of the hypothesis that $\mu$ is ergodic with respect to the holomorphic correspondence $\digamma$, we obtain $\mu (B) = 0$ or $1$. This implies $\mu_{E} (B) = 0$ or $1$, proving the ergodicity of the induced measure $\mu_{E}$ with respect to the induced correspondence $\digamma_{\hspace{-4pt} E}$. 
\end{proof} 

\section{Proofs of theorems on ergodicity}
\label{sec:erg} 

Connoisseurs of dynamical systems may know that the term ergodic is used to describe a dynamical system which, in a broad sense, exhibits equivalence between local time averages and the global space average; that is, the long-term behaviour of almost every orbit reflects the statistical properties of the entire space of states. In this section, which is independent of the preceding parts of the article, we undertake a detailed analysis of ergodicity in the context of holomorphic correspondences. 

A natural question arises when considering the two distinct notions of ergodicity - one associated with the $\sigma^{\Gamma}$ map and the other with the holomorphic correspondence $\digamma$: does ergodicity with respect to one imply the ergodicity with respect to the other? Before presenting the statement that answers this question, we recall a theorem from Walters \cite{walt:1982}, which will be useful in establishing an answer.

\begin{theorem} \cite{walt:1982}
\label{thm:walt_6.14}
Suppose $X$ is a compact metric space.  Let $T: X \longrightarrow X$ be a continuous map and $M$ be a $T$-invariant probability measure supported on $X$. Then $T$ is ergodic with respect to $M$ if and only if 
\[ \frac{1}{n} \sum_{i\, =\, 0}^{n - 1} \delta_{T^{i} (x)}\ \ \to\ \ M,\ \ \ \ \text{for}\ M\text{-almost every point}\ x \in X. \] 
\end{theorem}

\begin{proposition}
\label{prop:erg}
Let $\digamma$ be a holomorphic correspondence defined on $\widehat{\mathbb{C}}$ and $\mu \in \mathscr{E}^{\Gamma}$ be a $\digamma^{*}$-invariant measure supported on $\widehat{\mathbb{C}}$. Then $\mu$ is ergodic with respect to $\digamma$.   
\end{proposition}

\begin{proof} 
Suppose $\mu$ is a measure which satisfies the hypothesis of Proposition \ref{prop:erg}. Then, using the observation made in Section \ref{sec:kac_proof}, we can write $\mu = \left( \Pi_{0} \right)_{*} \mathcal{M}$ for some $\left( \sigma^{\Gamma} \right)^{-1}$-ergodic measure $\mathcal{M}$. Since the shift map $\left( \sigma^{\Gamma} \right)^{-1}$ is continuous on the compact metric space $\mathscr{O}^{\Gamma} \left( \widehat{\mathbb{C}} \right)$, by Theorem \ref{thm:walt_6.14} we have 
\[ \frac{1}{n} \sum_{i\, =\, 0}^{n - 1} \delta_{\left( \left( \sigma^{\Gamma} \right)^{-1} \right)^{i} \left( \mathfrak{X} \left( z_{0}; \boldsymbol{\gamma} \right)_{\boldsymbol{l}} \right)}\ \ \to\ \ \mathcal{M}\ \ \ \ \forall\ \ \mathfrak{X} \left( z_{0}; \boldsymbol{\gamma} \right)_{\boldsymbol{l}}\ \in\ A\ \subseteq\ \mathscr{O}^{\Gamma} \left( \widehat{\mathbb{C}} \right), \] 
where $\mathcal{M} (A) = 1$. This, in turn yields 
\[ \left( \Pi_{0} \right)_{*} \left( \frac{1}{n} \sum_{i\, =\, 0}^{n - 1} \delta_{\left( \left( \sigma^{\Gamma} \right)^{-1} \right)^{i} \left( \mathfrak{X} \left( z_{0}; \boldsymbol{\gamma} \right)_{\boldsymbol{l}} \right)} \right)\ \ \to\ \ \mu\ \ \ \ \forall\ \ \mathfrak{X} \left( z_{0}; \boldsymbol{\gamma} \right)_{\boldsymbol{l}}\ \in\ A\ \subseteq\ \mathscr{O}^{\Gamma} \left( \widehat{\mathbb{C}} \right). \]
Note that $\mu \left( \Pi_{0} (A) \right) = 1$. Subsequently,
\[ \lim_{n\, \to\, \infty} \frac{1}{n} \sum_{i\, =\, 0}^{n - 1} \left( f \circ \Pi_{0} \right) \left( \left( \left( \sigma^{\Gamma} \right)^{-1} \right)^{i} \left( \mathfrak{X} \left( z_{0}; \boldsymbol{\gamma} \right)_{\boldsymbol{l}} \right) \right)\ \ =\ \ \int_{\mathscr{O}^{\Gamma} \left( \widehat{\mathbb{C}} \right)} \left( f \circ \Pi_{0} \right) \left( \mathfrak{X} \left( z_{0}; \boldsymbol{\gamma} \right)_{\boldsymbol{l}} \right) \mathrm{d} \mathcal{M}, \]
for all points $z_{0} \in \Pi_{0} (A) \subseteq \widehat{\mathbb{C}}$, and for every function $f \in \mathcal{C} \left( \widehat{\mathbb{C}}, \mathbb{R} \right)$. 

Let $g \in \mathscr{L}^{1} \left( \widehat{\mathbb{C}}, \mu \right)$ be a bounded function defined on $\widehat{\mathbb{C}}$. Multiplying both sides of the above equation by $g \circ \Pi_{0}$, integrating the expressions with respect to the $\left( \sigma^{\Gamma} \right)^{-1}$-ergodic measure $\mathcal{M}$ and applying the dominated convergence theorem, we get 
\begin{eqnarray}
\label{eqn:6}
& & \lim_{n\, \to\, \infty} \bigintss_{\mathscr{O}^{\Gamma} \left( \widehat{\mathbb{C}} \right)} \left( \frac{1}{n} \sum_{i\, =\, 0}^{n - 1} \left( f \circ \Pi_{0} \right) \left( \left( \left( \sigma^{\Gamma} \right)^{-1} \right)^{i} \left( \mathfrak{X} \left( z_{0}; \boldsymbol{\gamma} \right)_{\boldsymbol{l}} \right) \right) \right) \left( g \circ \Pi_{0} \right) \left( \mathfrak{X} \left( z_{0}; \boldsymbol{\gamma} \right)_{\boldsymbol{l}}\right) \mathrm{d} \mathcal{M} \nonumber \\ 
& = & \left( \int_{\mathscr{O}^{\Gamma} \left( \widehat{\mathbb{C}} \right)} \left( f \circ \Pi_{0} \right) \left( \mathfrak{X} \left( z_{0}; \boldsymbol{\gamma} \right)_{\boldsymbol{l}}\right) \mathrm{d} \mathcal{M} \right) \cdot \left( \int_{\mathscr{O}^{\Gamma} \left( \widehat{\mathbb{C}} \right)} \left( g \circ \Pi_{0} \right) \left( \mathfrak{X} \left( z_{0}; \boldsymbol{\gamma} \right)_{\boldsymbol{l}} \right) \mathrm{d} \mathcal{M} \right). 
\end{eqnarray}

Further, consider an almost invariant set $E \subseteq \widehat{\mathbb{C}}$ with respect to $(\digamma, \mu)$. Making use of the density of continuous functions in the space of $\mathscr{L}^{1} \left( \mathscr{O}^{\Gamma} \left( \widehat{\mathbb{C}} \right), \mathcal{M} \right)$-functions, we can choose $f = \chi_{E}$ and $g = \chi_{E^{c}}$ in Equation \eqref{eqn:6} which reduces the right hand side of Equation \eqref{eqn:6} to $\mu(E) \cdot \mu(E^{c})$. Owing to the almost invariance of $E$ with respect to $(\digamma, \mu)$, 
\[ \frac{1}{n} \sum_{i\, =\, 0}^{n - 1} \left( \chi_{E} \circ \Pi_{0} \right) \left( \left( \left( \sigma^{\Gamma} \right)^{-1} \right)^{i} \left( \mathfrak{X} \left( z_{0}; \boldsymbol{\gamma} \right)_{\boldsymbol{l}} \right) \right)\ \ =\ \ \left( \chi_{E} \circ \Pi_{0} \right) \left( \mathfrak{X} \left( z_{0}; \boldsymbol{\gamma} \right)_{\boldsymbol{l}} \right) \] 
for $\mathcal{M}$-almost every $\mathfrak{X} \left( z_{0}; \boldsymbol{\gamma} \right)_{\boldsymbol{l}} \in \mathscr{O}^{\Gamma} \left( \widehat{\mathbb{C}} \right)$. The left hand side of Equation \eqref{eqn:6} reduces to $0$, resulting in $\mu (E) \cdot \mu \left( E^{{\rm c}} \right) = 0$. Thus, $\mu(E) \in \{ 0, 1 \}$ and hence $\mu$ is ergodic with respect to $\digamma$. 
\flushright\end{proof} 

\subsection{Proof of Theorem \ref{thm:equierg}}
We begin this subsection with a result due to Londhe, that shall come in handy in the proof. 

\begin{lemma} 
\label{lem:londhe}
\cite{lon:2022}
Let $\digamma$ be a holomorphic correspondence defined on $\widehat{\mathbb{C}}$ and $\mu \in \mathscr{M} \left( \widehat{\mathbb{C}} \right)$ be a $\digamma^{*}$-invariant measure that does not put any mass on polar sets. Then, the following statements are true. 
\begin{enumerate} 
\item For any Borel set $E \subseteq \widehat{\mathbb{C}}$, we have $\mu (E) \le \mu \left( \digamma^{\dagger} (E) \right)$. 
\item $\mu$ is ergodic with respect to the holomorphic correspondence $\digamma$ if and only for any $f \in \mathscr{L}^{1} \left( \widehat{\mathbb{C}}, \mu \right)$ satisfying $\displaystyle{\mathcal{T}_{\digamma} (f) = f}$ for $\mu$-almost every $z \in \widehat{\mathbb{C}}$, we have $f$ to be a constant $\mu$-almost everywhere. 
\end{enumerate} 
\end{lemma} 

\begin{proof}[of Theorem \ref{thm:equierg}] We prove the statements in the theorem one-by-one. 
\medskip 

{\bf \underline{1 $\Longrightarrow$ 2}}: Let $\mu$ be ergodic with respect to the holomorphic correspondence $\digamma$. Let $E$ be a Borel set for which $\mu (E) > 0$. Define $\displaystyle{E_{\infty} = \bigcup_{n\, \ge\, 1} \left( \digamma^{\dagger} \right)^{n} (E)}$. Then, $E_{\infty}$ is a Borel set that satisfies $\displaystyle{\digamma^{\dagger} \left( E_{\infty} \right) \subseteq E_{\infty}}$, meaning $E_{\infty}$ is an almost invariant set with respect to $\left( \digamma, \mu \right)$. Hence, by Definition \ref{defn:Gammaai}, $\mu \left( E_{\infty} \right) = 0$ or $1$. An application of statement (1) of Lemma \ref{lem:londhe} along with the definition of the set $E_{\infty}$ then gives us $\mu \left( E_{\infty} \right) = 1$. 
\medskip 

{\bf \underline{2 $\Longrightarrow$ 3}}: Since $E'$ is a Borel set of strict positive measure, we have 
\[ 0\ \ <\ \ \mu \left( E' \right)\ \ =\ \ \mu \left( E' \cap E_{\infty} \right)\ \ =\ \ \mu \left( \bigcup_{n\, \ge\, 1} \left( E' \cap \left( \digamma^{\dagger} \right)^{n} (E) \right) \right), \] 
that, in turn assures of the existence of some $n \in \mathbb{Z}_{+}$ for which $\displaystyle{\mu \left( E' \cap \left( \digamma^{\dagger} \right)^{n} (E) \right) > 0}$. 
\medskip 

{\bf \underline{3 $\Longrightarrow$ 1}}: Assuming the statement in $(3)$, we need to prove that $\mu$ is ergodic with respect to the holomorphic correspondence $\digamma$. In order to do so, we assume the contrary and arrive at a contradiction. Towards that end, consider an almost invariant Borel set $B$ with respect to $\left( \digamma, \mu \right)$ for which $0 < \mu (B) < 1$. Consider the set $B_{*}$, as defined in the proof of Lemma \ref{prop:usedintheorem}. As mentioned earlier, for every $k \in \mathbb{Z}_{+}$, we have $\displaystyle{\left( \digamma^{\dagger}\right)^{k} \left( B_{*} \right) \subseteq B}$ and $\displaystyle{\mu \left( B_{*} \right) = \mu (B)}$. Thus, 
\[ \mu \left( \left( \digamma^{\dagger}\right)^{k} \left( B_{*} \right) \cap B^{{\rm c}} \right) = 0\ \ \ \ \text{for every}\ k \in \mathbb{Z}_{+}, \] 
providing a contradiction to the hypothesis in statement (3), since both the sets $B^{c}$ and $ \left( \digamma^{\dagger}\right)^{k} \left( B_{*} \right)$ are of strict positive measure. 
\medskip 

{\bf \underline{1 $\Longrightarrow$ 4}}: By hypothesis, $\mu$ is ergodic with respect to the holomorphic correspondence $\digamma$. Thus, for any $f \in \mathscr{L}^{1} \left(\widehat{\mathbb{C}}, \mu \right)$, we have 
\[ \int_{\widehat{\mathbb{C}}} f \mathrm{d}\mu\ \ =\ \ \int_{\widehat{\mathbb{C}}} \sum_{w\, \in\, \digamma^{\dagger} (z)} \frac{f(w)}{d_{{\rm top}}} \mathrm{d}\mu\ \ =\ \ \int_{\widehat{\mathbb{C}}} \mathcal{T}_{\digamma} (f) \mathrm{d}\mu, \] 
that implies $\displaystyle{\int_{\widehat{\mathbb{C}}} \left( \mathcal{T}_{\digamma} (f) - f \right) \mathrm{d}\mu = 0}$. This observation along with the hypothesis $\displaystyle{\mathcal{T}_{\digamma} (f) - f \ge 0}$, $\mu$-almost everywhere, gives $\mathcal{T}_{\digamma} (f) = f,\ \mu$ almost everywhere. By an application of the second statement in Lemma \ref{lem:londhe}, we have $f \equiv c,\ \mu$ almost everywhere, where $c$ is some constant. 
\medskip 

{\bf \underline{4 $\Longrightarrow$ 1}}: Let $B$ be a Borel set that is almost invariant with respect to $\left( \digamma, \mu \right)$. Consider the characteristic function $\chi_{B}$. Then, we have 
\[ \mathcal{T}_{\digamma} \left( \chi_{B} \right) (z)\ \ =\ \ \sum_{w\, \in\, \digamma^{\dagger} (z)} \frac{\chi_{B} (w)}{d_{{\rm top}}}\ \ =\ \ \begin{cases} 1 & \text{for}\ \mu\text{-almost every}\ z \in B, \\ & \vspace{-10pt} \\ 0 & \text{for}\ \mu\text{-almost every}\ z \in B^{{\rm c}}. \end{cases} \] 
However, our assumption in statement (4) implies $\mu (B) = 0$ or $1$, proving the ergodicity of $\mu$ with respect to $\digamma$. 
\medskip 

{\bf \underline{1 $\Longleftrightarrow$ 5}}: The proof of $1 \Longrightarrow 5$ follows directly by appealing to statement (2) in Lemma \ref{lem:londhe}, while the proof of $5 \Longrightarrow 1$ follows by repeating the arguments in the proof of $4 \Longrightarrow 1$, as written above, since $\chi_{B} \in \mathscr{L}^{p} \left( \widehat{\mathbb{C}}, \mu \right)$ for any $1 \le p < \infty$. Both these are left as an exercise to the readers. 
\end{proof} 

Intuitively, the equivalences between the statements (1), (2) and (3) of Theorem \ref{thm:equierg} imply that no subset of positive size remains isolated under the action of the dynamics. Any set with positive measure eventually spreads throughout the whole space and any two such sets will overlap after some iteration. In essence, the dynamics forces all the regions in the space to interact with each other, making the space behave like a single, inseparable whole.

\subsection{Proof of Theorem \ref{thm:dense}}

In this subsection, we prove Theorem \ref{thm:dense} which talks about the density of images of a point in $\widehat{\mathbb{C}}$, in the classical sense, with respect to a holomorphic correspondence $\digamma$.

\begin{proof}[Proof of Theorem \ref{thm:dense}]
To prove the theorem, it suffices to show that for $\mu$-almost every point in $\widehat{\mathbb{C}}$, the forward image set $\displaystyle{I (z) = \bigcup_{i\, \ge\, 0} \digamma^{i} (z)}$ is dense in $\widehat{\mathbb{C}}$. Let $\{ B_{n} \}_{n\, \ge\, 1}$ be a countable base for the topology of the compact metric space $\widehat{\mathbb{C}}$. A point $z \in \widehat{\mathbb{C}}$ has its forward images to be dense in $\widehat{\mathbb{C}}$ if and only if for every $n \in \mathbb{Z}^{+}$ there exists $i \ge 0$ such that  $\digamma^{i} (z) \cap B_{n} \ne \emptyset$. In other words, $\displaystyle{z \in \bigcap_{n\, \ge\, 1} \bigcup_{i\, \ge\, 0} \left( \digamma^{i} \right)^{\dagger} \left( B_{n} \right)}$. For each $n \in \mathbb{Z}^{+}$, define $\displaystyle{A_{n} = \bigcup_{i\, \ge\, 0} \left( \digamma^{i} \right)^{\dagger} \left( B_{n} \right)}$.  By construction, each $A_{n}$ is measurable and a straightforward verification shows that $A_{n}$'s are almost invariant with respect to $(\digamma, \mu)$, for every $n \in \mathbb{Z}^{+}$. By the ergodicity of the measure $\mu$ with respect to $\digamma$, it follows that $\mu \left( A_{n} \right) \in \{ 0, 1 \}$ for all $n \in \mathbb{Z}^{+}$. Since $\mu \left( B_{n} \right) > 0$ by hypothesis and $B_{n} \subseteq A_{n}$, we have $0 < \mu \left( B_{n} \right) \le \mu \left( A_{n} \right)$ which forces $\mu \left( A_{n} \right) = 1$ for all $n \in \mathbb{Z}^{+}$. Consequently, we have $\displaystyle{\mu \left( \bigcap_{n\, \ge\, 1} A_{n} \right) = 1}$ which concludes the proof.
\end{proof}

\subsection{Proof of Theorem \ref{thm:rkforc}}

Prior to presenting the proof of Theorem \ref{thm:rkforc}, we recall Kakutani-Rokhlin's lemma for maps, that plays a key role in our subsequent analysis.

\begin{lemma} \cite{ew:2011}
\label{lem:kak-rok}
Let $(X, \mathscr{B}, M, T)$ be an invertible ergodic measure-preserving system and assume that $\mathcal{M}$ is non-atomic. Then for any $n \geq 1$ and $\epsilon > 0$, there exists a set $B \in \mathscr{B}$ with the property that $B, T (B), \cdots, T^{n - 1} (B)$ are disjoint sets and $M \left( B \sqcup T (B) \sqcup \cdots \sqcup T^{n - 1} (B) \right) > 1 - \epsilon$.
\end{lemma}

\begin{proof}[Proof of Theorem \ref{thm:rkforc}]
Consider the space of bi-infinite sequences $\mathscr{O}^{\Gamma} \left( \widehat{\mathbb{C}} \right)$, associated with the holomophic correspondence $\Gamma$, as defined in Equation \eqref{eqn:di_gamma}, endowed with the Borel $\sigma$-algebra $\mathscr{B}$. Let $\mathcal{M}$ be a non-atomic, ergodic measure supported on $\mathscr{O}^{\Gamma} \left( \widehat{\mathbb{C}} \right)$. Then by Lemma \ref{lem:kak-rok} for any $n \geq 1$ and $\epsilon > 0$, there exists a measurable set $B \subseteq \mathscr{O}^{\Gamma} \left( \widehat{\mathbb{C}} \right)$ such that the sets $B, \sigma^{\Gamma} (B), \cdots, \left( \sigma^{\Gamma} \right)^{n - 1} (B)$ are pairwise disjoint and satisfy 
\[ \mathcal{M} \left( B \sqcup \sigma^{\Gamma} (B) \sqcup \cdots \sqcup \left( \sigma^{\Gamma} \right)^{n - 1} (B) \right)\ \ >\ \ 1 - \epsilon. \] 

Define the set $E = \Pi_{0} (B) \subseteq \widehat{\mathbb{C}}$ and let $\mu$ be the measure as specified in the hypothesis of the Theorem \ref{thm:rkforc}. Then, an application of the measurable projection theorem ensures that $E \in \mathscr{B_{\mu}}$. Note that for any $k \in \mathbb{Z}_{+} \cup \{ 0 \}$ we have $\Pi_{0} \left( \left( \sigma^{\Gamma} \right)^{k} (B) \right) \subseteq \digamma^{k}(E)$ which gives 
\[ \Pi_{0} \left( \bigsqcup\limits_{i\, =\, 0}^{n - 1} \left( \sigma^{\Gamma} \right)^{i} (B) \right)\ \ =\ \ \bigcup\limits_{i\, =\, 0}^{n - 1} \Pi_{0} \left( \left( \sigma^{\Gamma} \right)^{i} (B) \right)\ \ \subseteq\ \ \bigcup\limits_{i\, =\, 0}^{n - 1} \digamma^{i} (E). \] 
Consequently, we obtain
\[ \mu \left( \bigcup\limits_{i\, =\, 0}^{n - 1} \digamma^{i} (E) \right)\ \ \ge\ \ \mu \left( \Pi_{0} \left( \bigcup\limits_{i\, =\, 0}^{n - 1} \left( \sigma^{\Gamma} \right)^{i} (B) \right) \right)\ \ \ge\ \ \mathcal{M} \left( \bigsqcup\limits_{i\, =\, 0}^{n - 1} \left( \sigma^{\Gamma} \right)^{i} (B) \right)\ \ >\ \ 1 - \epsilon, \] 
thereby completing the proof. 
\end{proof}

Observe that we cannot comment about the disjointedness of the sets $B, \digamma (B), \cdots, \digamma^{n - 1} (B)$ using the above technique, since $X \cap Y = \emptyset$ does not imply $\Pi_{0} (X) \cap \Pi_{0} (Y) = \emptyset$, where $X, Y \subseteq \mathscr{O}^{\Gamma}\left( \widehat{\mathbb{C}} \right)$.

\end{document}